\documentclass[12pt]{amsart}
\usepackage[margin=1in]{geometry}
\usepackage{comment}
\usepackage{mathrsfs}
\usepackage{latexsym}
\usepackage{float}
\usepackage{amssymb}
\usepackage[all]{xy}
\usepackage{epsfig}
\usepackage{color}
\usepackage{graphics}
\usepackage{hyperref}

\newtheorem{Thm}[equation]{Theorem}
\newtheorem{Lem}[equation]{Lemma}
\newtheorem{Cor}[equation]{Corollary}

\begin{document}
\baselineskip=12pt

\title{A classification of rational 3-tangles }
\author{Bo-hyun Kwon}
\date{Friday, May 23, 2025}
\maketitle
\begin{abstract} 

 In this paper, we define the \textit{normal form} and \textit{normal coordinate} of  a rational 3-tangle $T$ with respect to $\partial E_1$, where $E_1$ is the fixed two punctured disk in $\Sigma_{0,6}$. Among all normal coordinates of $T$  with respect to $\partial E_1$, we  investigate the collection of \textit{minimal} normal coordinates of $T$.
We show that the simplicial complex constructed with normal forms of the rational 3-tangle is contractible. As an effectiveness of  the contractibility of the simplicial complex by normal forms of $T$, we would choose  a minimal normal coordinate of $T$  with a certain rule for the representative for the rational $3$-tangle $T$. This classifies rational $3$-tangles up to isotopy.
\end{abstract}

\section{Introduction}

J. Conway~\cite{0} developed Conway's Theorem that gives a perfect classification of rational $2$-tangle with the correspondence between rational 2-tangles and rational numbers. It states that \textit{two rational (2-)tangles
are ambient isotopic if and only if their fractions are equal.} The definition of the \textit{rational} tangles satisfied a certain property among $2$-tangles does make sense due to the correspondence with rational numbers. However, unfortunately, we were not able to find a certain meaningful relation between rational $n$-tangles and rational numbers if $n\geq 3$ so far. Some topologists call  the rational tangles as \textit{trivial} tangles although trivial  in  knot diagrams means that there is no crossing up to isotopy.
Over 50 years, we were not able to extend the classification of rational $2$-tangles to a classification of rational $n$-tangles($n>2$) properly, especially rational $3$-tangles in advance. Cabrera~\cite{2}  found a pair of invariants which is defined for all rational $3$-tangles and it classifies a special set of rational $3$-tangles containing the $3$-braid group. The author~\cite{1} gave an algorithm to distinguish arbitrary two rational $3$-tangles. Currently, we defined an \textit{arc systems} of a rational $3$-tangle $T$ and made a nice operator called \textit{bridge arc replacement} which preserve the rational tangle $T$ but change the arc system of $T$. This operator is very useful tool to detect the trivial rational $3$-tangle called $\infty$ tangle. In the argument for the algorithm we developed in~\cite{1}, the main difficulty to make the algorithm was  how to check whether or not the given rational $3$-tangle is the $\infty$ tangle. The result by the bridge arc replacement gives much easier criterion to detect $\infty$ tangle. The author~\cite{3} defined  \textit{normal forms} for rational $3$-tangles. We also showed that there exists a sequence of movements called \textit{normal jump moves} that leads an arbitrary normal form to another arbitrary normal form. In this paper, we show that we can select a representative for normal forms of a given rational $3$-tangle by regarding minimality of \textit{weights}. Here, weights mean the numbers of intersections with the three simple closed curves for the standard formation to define normal forms. In order to show this, we define a \textit{complex of normal forms}
 of a rational $3$-tangle in Section~\ref{s4}.
We show that the complex of normal forms of a rational $3$-tangle is contractible in Section~\ref{s5}. Among all of normal forms of a rational $3$-tangle, we first consider normal forms with minimal intersections in terms of  $\partial E$ which consists of the disjoint non-parallel three simple closed curves enclosing two punctures respectively. In Section~\ref{s6}, we discuss a specific regulation to choose one of them for the representative  to have our classification.

\section{Arc system of a rational 3-tangle}

\subsection{Rational n-tangle and equivalence of  rational $3$-tangles}

Let $\tau=\tau_1\cup\tau_2\cup\cdot\cdot\cdot\cup \tau_n$ be a set of $n$ pairwise disjoint arcs $\tau_i$ properly embedded in a $3$-ball $B$. Then, the pair $(B,\tau)$ is called an $n$-tangle. An $n$-tangle is \textit{rational} if there exists a homeomorphism of pairs $H$ from $(B,\tau)$ to $(D^2\times I, \{p_1,p_2,...,p_n\}\times I)$, where $D^2$ is a $2$-dimensional unit disk, $I$ is a unit interval and $p_i\in$ int~$D^2$ for all $i\in\{1,2,...,n\}$. In order to classify rational $n$-tangles, we would fix the $2n$ endpoints of the $n$ arcs. We say that two rational $n$-tangles $(B,\tau)$ and $(B,\tau')$ are $\textit{isotopic}$ (or equivalent) if there exists a homeomorphism of pairs $K$ from $(B,\tau)$ to $(B,\tau')$ so that $K|_{\partial B}=$id.
 We project $B$ onto $xy$-plane bounded by a great circle $C$. Let $D$ be the disk in $xy$-plane bounded by $C$. We fix the endpoints of $n$ arcs in $C$ regularly. We assume that the projection of $\tau$ onto $D$ is in minimal general position. Moreover, throughout this paper, we assume that the intersections between graphs are in minimal general position. By considering over- and under crossing such as a knot diagram, we have a tangle diagram in $D$. For the classification of rational $2$-tangles, we mainly investigate crossing patterns and relations between the strings directly in the diagrams. However, there is no meaningful pattern from the diagrams of rational $n$-tangles to classify rational $n$-tangle if $n>2$ so far. By the definition of rationality of $n$-tangles, we note that there is a projection of $\tau$ onto the boundary of the $3$-ball simultaneously. In other words, we have a projection of $\tau$ onto $\partial B$ so that the projection of each arc of $\tau$ is a simple arc in $\partial B$ and there is no intersection between the projections of $\tau$ in $\partial B$. Let $\beta_i$ be the projection of $\tau_i$. The projection can guarantee the existence of disjoint three disks in $B$ which are bounded by $\tau_i\cup\beta_i$ respectively. The disks are called a \textit{bridge disk} and $\beta_i$ are called a \textit{bridge arc}. We  note that the disjoint $n$ bridge arcs in $\partial B$ determine the rational $n$-tangle. So, instead of tangle diagrams, we investigate the patterns of disjoint $n$ bridge arcs in six punctured sphere $\Sigma_{0,2n}$. A collection of disjoint $n$ bridge arcs in $\Sigma_{0,2n}$ is called an \textit{arc system} of the rational $n$-tangle $(B,\tau)$. In this paper, we investigate  disjoint three bridge arcs in $\Sigma_{0,6}$ to classify rational $3$-tangles. In order to parameterize simple closed curves in $\Sigma_{0,6}$, we take the three $2$-punctured disks $E_1,E_2$ and $E_2$ which contain the projections of the three strings of the trivial rational $3$-tangle as in Figure~\ref{f0} which is called $\infty$ tangle. Then the following theorem gives a classification of simple arcs in $\Sigma_{0,6}$ in terms of $E_1,E_2$ and $E_3$.

\begin{figure}[htb]
\includegraphics[scale=.3]{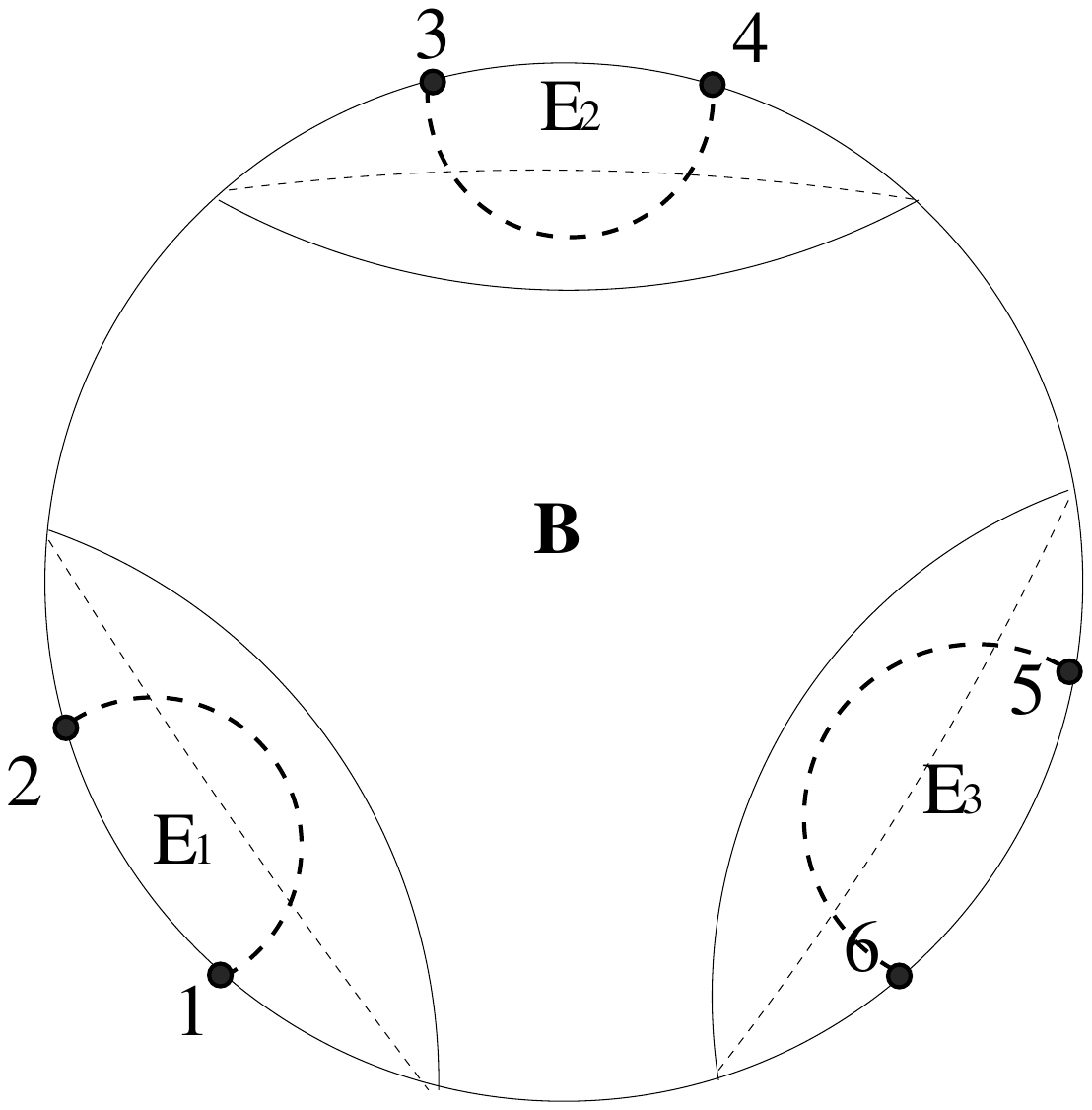}
\vskip -70pt
\caption{Two punctured disks $E_1, E_2$ and $E_3$}\label{f0}
\end{figure}

\begin{Thm}[Special case II of Dehn's theorem]\label{T-1} There is a one-to-one map $\phi :\mathcal{C}\rightarrow \mathbb{Z}^6$ so that $\phi([\beta])=
(p_1,q_1,p_2,q_2,p_3,q_3)$, where $\beta$ is a simple arc connecting two  punctures in $\Sigma_{0.6}$ and $\mathcal{C}$ is the collection of isotopy classes of simple arcs.  i.e., it classifies isotopy classes of simple arcs (bridge arcs).
\end{Thm}

In Theorem~\ref{T-1}, $p_i$ is the number of intersections between $\beta$ and $E_i$ and $q_i$ stands for the connecting pattern between the intersections of $\beta$ with $\partial E_i$ in $E_i$. Refer to the subsection~\ref{2.2} for details. We note that the intersection numbers $p_1,p_2$ and $p_3$ determine  all the arc components in $\Sigma_{0.6}\setminus E$, where $E=E_1\cup E_2\cup E_3$. We also note that a number of disjoint simple arcs also can be classified by Theorem~\ref{T-1}.

\subsection{Dehn's parametrization of an arc system of  a rational 3-tangle}\label{2.2}

Suppose that $\{\beta_1,\beta_2,\beta_3\}$ be an arc system of a rational tangle $T$. Let $1$ and $2$ be the two punctures in $E_1$ as in Figure~\ref{f0}. We define a subarc $\omega_1$ of $\partial E_1$ which contains all intersections between $\beta(=\beta_1\cup\beta_2\cup\beta_3)$ and $\partial E_1$ which is called the \textit{window} of $E_1$.
 We take a barycentric disk $E_1'$ in $E_1$ as in the first diagram of Figure~\ref{f1}. In order to find the Dehn's parameters for $\beta_1,\beta_2$ and $\beta_3$ in $E_1$, we assume that the three bridge arcs pass through $E_1'$ as follows. All of the arc components between the three bridge arcs and $E_1'$ are perpendicular to $e_1$ if they intersect with $e_1$ as in the second diagram of Figure~\ref{f1}, where the equator $e_1$ is the straight line connecting the two punctures $1$ and $2$. The arc component connecting the puncture $1$ is perpendicular to $e_1$ below the puncture $1$ and the arc component connecting the puncture $2$ is perpendicular to $e_1$ above the puncture $2$ as in the second diagram of Figure~\ref{f1}. Let $b_1,b_2,...,b_{2n+2}$ be the sequence of intersections between the arc components and $\partial E_1$ from the leftmost intersection of the bottom to the direction of counter-clockwise as in the second diagram of Figure~\ref{f1}. Then, we connect the intersections of $\beta$ with $\partial E_1$ and $b_1,b_2,...,b_{2n+2}$. Suppose that the leftmost intersection of $\beta$ with $\omega_1$  connects $b_i$. There are infinitely many different simple arcs in $E_1$ which connect the leftmost intersection and $b_i$. However, we can formulate the connecting arc by using a twisting number $n$. If the connecting arc rotates $E_1'$ $k$ times completely, we assign $k$ as the twisting number. We note that there are two directions to rotate. If the connecting arc rotates counter clockwise then $k$ should be a non-negative integer. Otherwise, $k$ should be a non-positive integer. Then we note that all connecting patterns between the intersections in $\omega_1$ and $b_1.b_2,...,b_{2n+2}$ are determined since there is no intersection between the arc components in $E_1$ which connect the intersections in $\omega$ and $b_1,b_2,...,b_{2n+2}$. We choose $q_1=k\times(2n+2)+i-1$, where $k$ is an arbitrary integer and $1\leq i\leq n$. Similarly, we can decide $q_2$ and $q_3$ by considering the connecting patterns in $E_2$ and $E_3$ respectively. The ordered sequence of integers $(p_1,q_1,p_2,q_2,p_3,q_3)$ parameterizes the arc system $\{\beta_2,\beta_2,\beta_3\}$ of $T$.

\begin{figure}[htb]
\begin{center}
\includegraphics[scale=.8]{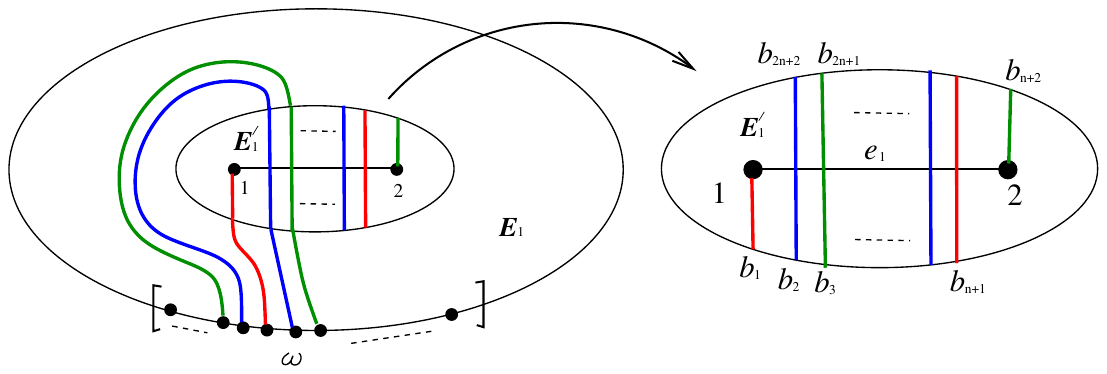}
\end{center}
\vskip -530pt
\caption{Barycentric disk $E_1'$}\label{f1}
\end{figure}

\section{Normal forms and normal jump moves}

\subsection{Normal forms of systems of rational 3-tangles }
Recall that $T=(B^3,\tau)$ is a rational $3$-tangle, $
\mathcal{B}=\{\beta_1,\beta_2,\beta_3\}$ is  an arc system of $T$ and $\beta=\beta_1\cup\beta_2\cup\beta_3$.  We 
say that an arc system $\mathcal{B}$ of $T$ is a \textit{normal form} with respect to $\partial E$ if there is no adjacent 
two intersections between $\beta$ and $\omega_i$ in $\omega_i$ which belong to the same $\beta_j$ for some $j\in\{1, 2, 3\}$. We note that there is no adjacent two intersections between $\beta$ and $e_i$ in $e_i$ as well  if $\mathcal{B}$ is a normal form as in the second diagram of Figure~\ref{f1}.
We assume that there exists adjacent two intersections between $\beta$ and $\omega_i$ in $\omega_i$ which belong to the same $\beta_j$ for some $j\in\{1,2,3\}$. Then we consider the subarc $\delta$ of $\omega_i$ so that the two endpoints of $\delta$ belong to a bridge arc different from $\beta_j$ and $\delta$ intersects only $\beta_j$. We note that $\delta$ cuts $\beta_j$ into at least three subarcs. By connecting the two subarcs of them which contain one of the endpoint of $\beta_j$ along $\delta$, we can construct a new bridge arc $\beta_j'$. We note that $\beta_j'$ is also disjoint with the other two bridge arcs different from $\beta_j$. We also note that any simple arc disjoint with the given two disjoint bridge arcs which connects the remaining two endpoints of $\Sigma_{0,6}$ is also a bridge arc.
The process to have the new bridge arc $\beta_j'$ in stead of $\beta_j$ is called the \textit{bridge arc replacement}.
 The bridge arc replacement can guarantee the existence of a normal form of $T$. (Refer to~\cite{3}.)

 \subsection{Normal jump moves} For  the arc system $\mathcal{B}=\{\beta_1,\beta_2,\beta_3\}$ of a rational $3$-tangle $T$,  we define the move of an element of the arc system called a \textit{jump move} (or sliding move) as follows. Take a regular neighborhood of $\beta_i$, named $N(\beta_i)$, in $\partial B$ which is a  disjoint with $\beta_j$ and $\beta_k$, where $\{i,j,k\}=\{1,2,3\}$.
Now, we consider a band $R$ so that the interior of $R$ is disjoint with $N(\beta_i), \beta_j$ and $\beta_k$ and two opposite sides of $R$ are subarcs of $N(\beta_i)$ and $\beta_j$ respectively as in the first diagram of Figure~\ref{c5}. We  now construct a new bridge arc by using the other two sides of $R$ and the rest of  $N(\beta_i)$ and $\beta_j$ excluding the subarcs for the two opposite sides of $R$  as in the diagrams of Figure~\ref{c5}. The movement to have the new bridge arc  is called the \textit{jump move} of $\beta_j$  (over $\beta_i$). We note that the bridge arc obtained by a bridge arc replacement also can be obtained by a jump move.
\begin{figure}[htb]
\includegraphics[scale=.7]{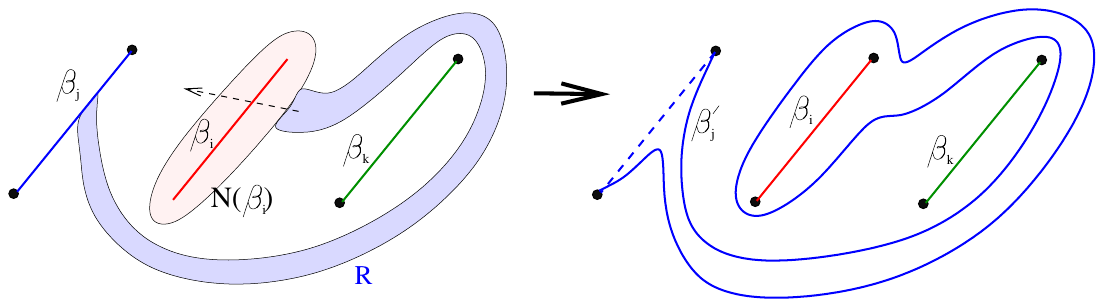}
\vskip -470pt
\caption{A jump move}\label{c5}
\end{figure}

The jump move is called a \textit{normal jump move} if the new arc system $\mathcal{B}'$  obtained from a normal arc system $\mathcal{B}$  by a jump move  is also a normal form.   The author~\cite{3} defined the \textit{standard normal jump move} which can guarantee the existence of normal jump move. The following theorem states the maximal number of normal jump moves.
\begin{Thm}(Theorem $8$~ in \cite{3})\label{T-2} Let $\mathcal{B}=\{\beta_1,\beta_2,\beta_3\}$ be a normal form for a rational $3$-tangle $T$.
There are at most two normal jump moves of $\beta_k$, where $k\in\{1,2,3\}$.
\end{Thm}

 By Theorem~\ref{T-2}, there are two possible normal jump moves of $\beta_i$. Let $\beta_i'$ be the obtained bridge arc by a normal jump move of $\beta_i$. We note that $\beta_i'$ can have the same intersection pattern with $\partial E$ having the intersections with $\beta_j$ and $\beta_k$ as $\beta_i$ had, where $\{i, j,k\}=\{1,2,3\}$. If not, then $\beta_i'$ should have the intersection pattern as follows.  (Refer to the arguments in Theorem~\ref{T-2} of ~\cite{3}.)\\

 Let $k_i$ be the normal jump move of $\beta_i$. We note that the intersection between $\beta_j$ and $\omega$ is preserved when we apply $k_i$ if $i\neq j$, where $\omega=\omega_1\cup\omega_2\cup\omega_3$. Then $\beta_i'$ satisfies the followings.

\begin{enumerate}
\item If the intersection of $\beta_i$ and $\omega$ has the  adjacent intersections with the same color in $\omega$ then the intersection is preserved when we replace $\beta_i$ by $\beta_i'$.
\item If the intersection of $\beta_i$ and $\omega$ has the adjacent intersections with different color in $\omega$ then the intersection disappears when we replace $\beta_i$ by $\beta_i'$.
\item If two adjacent intersections between $\beta$ and $\omega$ are not a part of $\beta_i$ then an intersection appears between the two adjacent intersections when we replace $\beta_i$ by $\beta_i'$.
\end{enumerate}\vskip 10pt

Let $k_i^1$ be the \textit{standard} normal jump move of $\beta_i$ which always exists if $\beta_i\cap \omega\neq \emptyset$. (Refer to  ~\cite{3}.) Let $k_i^2$ be  the other normal jump move of $\beta_i$ if it exists.
Figure~\ref{outlet} gives an example of the case that has $k_1^2$. In order to have $k_i^2$, we note that $\beta=(\beta_1\cup\beta_2\cup\beta_3)$ should have four outlets as in the diagram of Figure~\ref{outlet}.
\\

\begin{figure}[htb]
\includegraphics[scale=0.9]{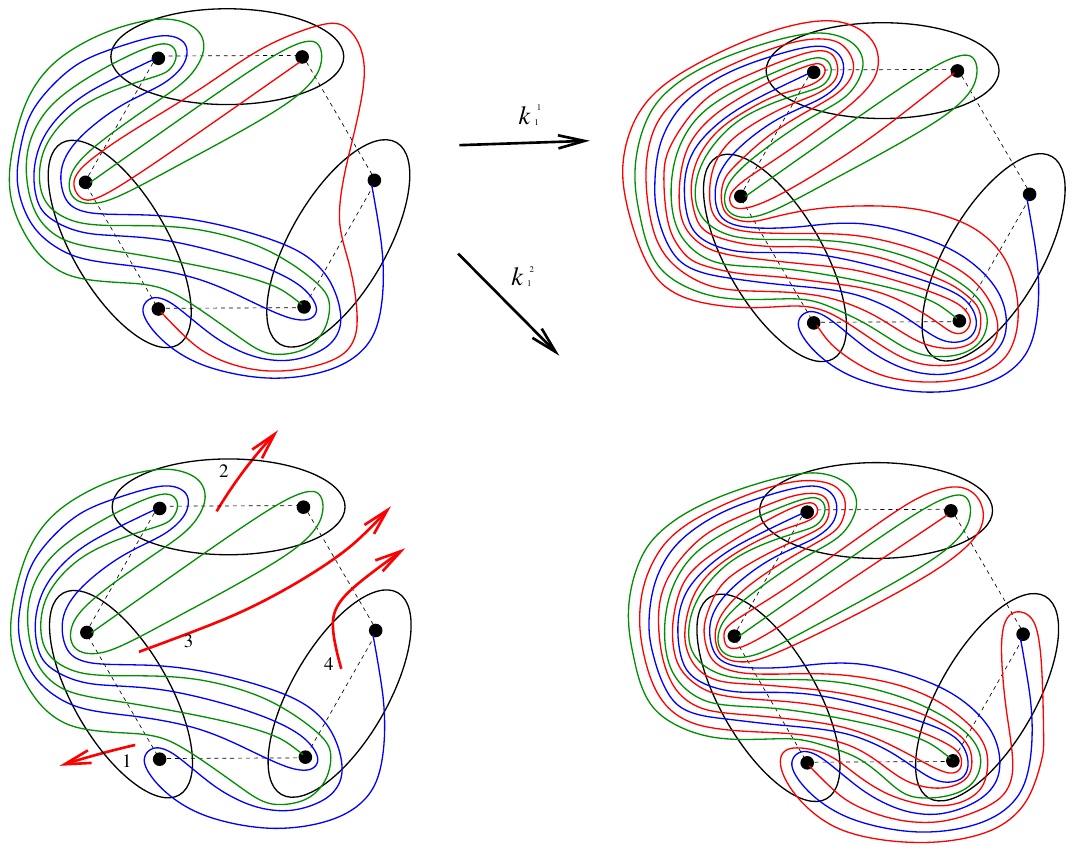}
\vskip -370pt
\caption{A diagram with $k_i^2$ }\label{outlet}
\end{figure}

\begin{Lem}\label{L0}
For a given normal form $\mathcal{B}=\{\beta_1,\beta_2,\beta_3\}$ for a rational $3$-tangle $T$, suppose that there exists $k_1^2$ for some $i$. Let $\{\beta_1',\beta_2,\beta_3\}$ be the normal form of $T$ by $k_1^2$. Then there is no $k_j^2$ from $\{\beta_1',\beta_2,\beta_3\}$ if $j\neq 1$.
\end{Lem}
\begin{proof}
\begin{figure}[htb]
\includegraphics[scale=0.9]{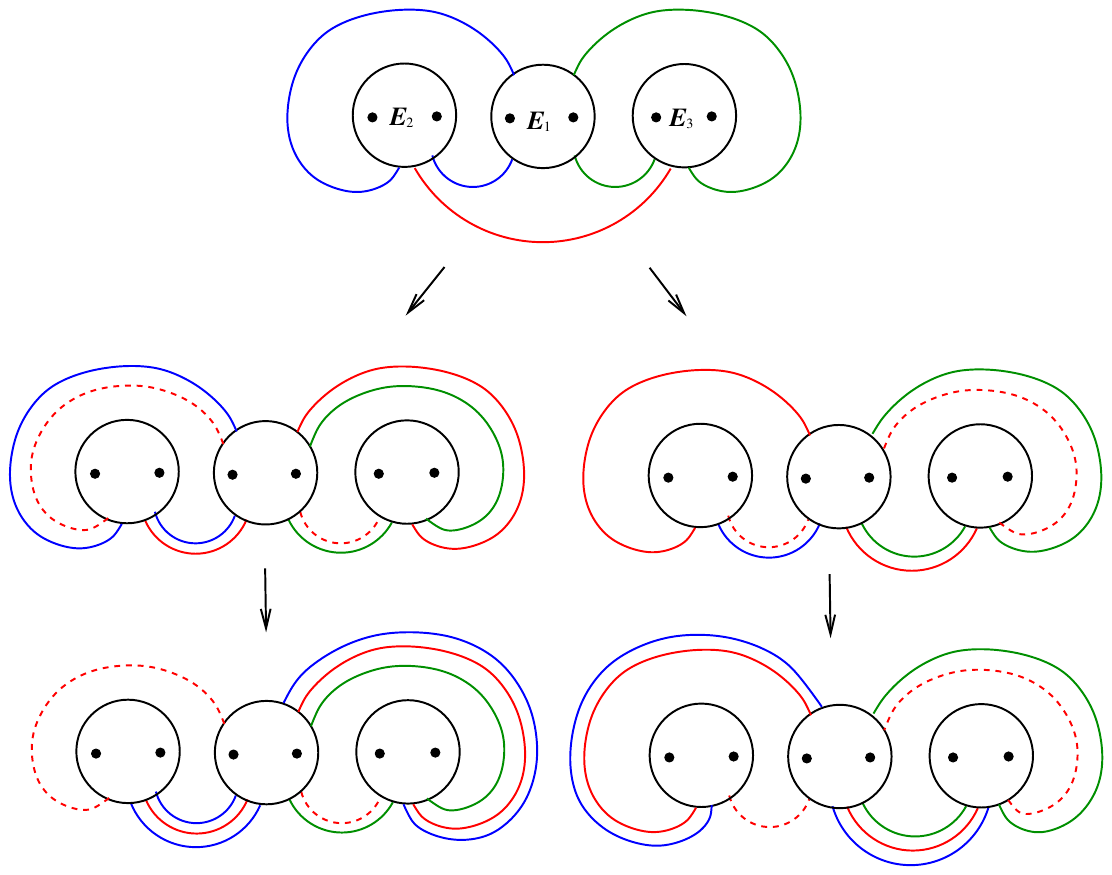}
\vskip -380pt
\caption{Outlets having $k_i^2$ }\label{outlet2}
\end{figure}

The upper diagram is a schematic diagram to have $k_i^2$. We want to point out that the given arcs with  green or blue color describe outermost arcs from $E_s$ to $E_t$ for distinct $s,t\in\{1,2,3\}$. By replacing the red bridge arc, we have the two normal jump moves depicted below. We note that the upper diagram is the only case to have $k_i^2$. If we ignore the red arc, we should have four outlets to have $k_i^2$. Also, the two adjacent intersections of the two endpoints of the red arc should belong to the same bridge arc. Also, we note that the bridge arcs for the adjacent intersections   are different since $\mathcal{B}$ is a normal form. Now, we investigate the lower two diagrams. We consider the blue arcs and the green arcs to check the existence of $k_j^2$ for $j\neq i$. As we mentioned above, the diagram ignoring the blue arcs (or green arcs) should have two outlets having two exit intersections belonging to the same bridge arc. Then either there is no such outlet or there are two outlets but not possible to have $k_j^2$. For instance, there is only one normal jump move replacing the given blue arc since  added blue arcs between the other two outlets to the given diagram make a simple closed curve. It makes a contradiction to have a normal jump move. Therefore, it is impossible to have $k_j^2$ after applying $k_i^2$ to the upper diagram of Figure~\ref{outlet2} if $j\neq i$. This completes the proof of this theorem.
\end{proof}

\begin{Lem}\label{L-2}
Let $\mathcal{B}=\{\beta_1,\beta_2,\beta_3\}$ be a normal form for a rational $3$-tangle $T$. Suppose that  subarcs of $\beta_1\cup\beta_2\cup\beta_3$ connect $E_i$ and $E_j$ for all distinct $i,j\in\{1,2,3\}$. Then there is no $k_s^2$ for any $s\in\{1,2,3\}$.
\end{Lem}

\begin{proof}
Let $\beta_i'$ be the bridge arc by $k_i$.
We note that there is only one outlet for $\beta_i'$ having the two adjacent intersections of it with the same color. By the regulation of  normal jump moves mentioned before, there exists unique normal jump move. We note that there exists one normal jump move then it should be $k_i^1$.
\end{proof}

\begin{Lem}\label{L-1}
For a given normal form $\mathcal{B}=\{\beta_1,\beta_2,\beta_3\}$ for a rational $3$-tangle $T$, let $\{\beta_1',\beta_2,\beta_3\}$ be the normal form by a normal jump move of $\beta_1$.  Suppose that $\beta_i\cap \partial E_1\neq \emptyset$ for all $i\in\{1,2,3\}$ and there exist distinct $s,t\in\{1,2,3\}$ so that there is no subarc of $\beta_1\cup\beta_2\cup\beta_3$ from $E_s$ to $E_t$. If $|\partial E\cap(\beta_1\cup\beta_2\cup\beta_3)|<|\partial E\cap(\beta_1'\cup\beta_2\cup\beta_3)|$ then $|\partial E_1\cap(\beta_1\cup\beta_2\cup\beta_3)|<|\partial E_1\cap(\beta_1'\cup\beta_2\cup\beta_3)|$.
\end{Lem}
\begin{proof}
We note that there exist a subarc of $\beta_1\cup\beta_2\cup\beta_3$ from $E_i$ to $E_j$ for distinct $i, j$ if $\{i,j\}\neq \{s,t\}$. So, there are two outlets for $\beta_1'$ having two adjacent intersections with the same color. Since $|\partial E\cap(\beta_1\cup\beta_2\cup\beta_3)|<|\partial E\cap(\beta_1'\cup\beta_2\cup\beta_3)|$, the intersection patterns of $\beta_1\cup\beta_2\cup\beta_3$ and $\beta_1'\cup\beta_2\cup\beta_3$ with $\partial E_1$ are different. So, $\beta_1'$ cannot be obtained from $k_1^2$. Then we note that the intersections of $\beta_1$ with $\partial E_1$ exist only in the middle of two adjacent intersections of $\beta_2\cup\beta_2$ with $\partial E_1$ belonging to the same bridge arc. By the assumption that $\beta_k\cap\partial E_1\neq \emptyset$ for all $k\in\{1,2,3\}$, there exists a pair of adjacent intersections of $\beta_2\cup\beta_3$ with $\partial E_1$ belonging to different bridge arcs. This implies that we also have  $|\partial E_1\cap(\beta_1\cup\beta_2\cup\beta_3)|<|\partial E_1\cap(\beta_1'\cup\beta_2\cup\beta_3)|$ by the regulation of normal jump moves and it completes the proof. (Refer to the diagrams of Figure~\ref{outlet}.)
\end{proof}

\begin{Lem}\label{L}
For a given normal form $\mathcal{B}=\{\beta_1,\beta_2,\beta_3\}$ for a rational $3$-tangle $T$, let $\{\beta_1',\beta_2,\beta_3\}$ be the normal form by a normal jump move of $\beta_1$.  Suppose that $|\partial E_1\cap(\beta_1\cup\beta_2\cup\beta_3)|>|\partial E_1\cap(\beta_1'\cup\beta_2\cup\beta_3)|$. Then there exists unique normal jump move of $\beta_1$ from $\mathcal{B}$ which satisfies the inequality.
\end{Lem}
\begin{proof}
We can consider the following three cases for this theorem.\\

\textbf{Case 1} : $\beta_i\cap \partial E =\emptyset$ $\forall i$. We note that  $\mathcal{B}$ stands for $\infty$ tangle and it is unique normal form of it. Actually, there is no  normal jump move that satisfies   $|\partial E_1\cap(\beta_1\cup\beta_2\cup\beta_3)|>|\partial E_1\cap(\beta_1'\cup\beta_2\cup\beta_3)|$. \\

\textbf{Case 2} : $(\beta_1\cup\beta_2\cup \beta_3)\cap \partial E \neq \emptyset$  but  $\beta_i\subset E_i$ for some $i$.  In order to have the normal form $\{\beta_1,\beta_2,\beta_3\}$,  we note that $\beta_j\cap \partial E \neq \emptyset$ if $j\neq i$.  If $i=1$ then there is no normal jump move replacing $\beta_1$ since $\partial E_1\cap(\beta_1\cup\beta_2\cup\beta_3)=\emptyset$. So, we assume that that $i\neq 1$. Then we have unique normal jump move, but $|\partial E_1\cap(\beta_1\cup\beta_2\cup\beta_3)|=|\partial E_1\cap(\beta_1'\cup\beta_2\cup\beta_3)|$. (Refer to~\cite{3}.)\\

\textbf{Case 3} : $\beta_i\cap \partial E \neq\emptyset$ $\forall i$. We can check that there is unique normal jump move if $|\partial E_1\cap(\beta_1\cup\beta_2\cup\beta_3)|>|\partial E_1\cap(\beta_1'\cup\beta_2\cup\beta_3)|$ since there are only two outlets of $\beta_1'$ in this case. We note that there is unique way to connect the subarcs from the two outlets up to isotopy. (Refer  to the argument in Theorem 8 in~\cite{3}.)

\end{proof}

\begin{Lem}\label{L1}
Let $\mathcal{B}=\{\beta_1,\beta_2,\beta_3\}$  be a normal form  for  a rational $3$-tangle $T$. Suppose that there exist both $k_i^1$ and $k_i^2$ for some $i$. Then $|\partial E_1\cap(\beta_1\cup\beta_2\cup\beta_3)|<|\partial E_1\cap  k_i^1(\beta_1\cup\beta_2\cup\beta_3)|$ and $|\partial E_1\cap(\beta_1\cup\beta_2\cup\beta_3)|<|\partial E_1\cap  k_i^2(\beta_1\cup\beta_2\cup\beta_3)|$.
\end{Lem}
\begin{proof}
Assume that  there exist $k_i^1$ and $k_i^2$ for some $i$. By referring to the argument in Theorem 7 in~\cite{3}, we note that there are four outlets of $\beta_i'$ which is obtained from $\beta_i$ by $k_i^1$ or $k_i^2$ . Two of them have exit intersections with same color and two of them have exist intersections with different colors, where the exist intersection means that the adjacent intersection of $\beta_1'$ for the outlets. In this case, we note that the standard normal jump move from $\mathcal{B}$ which has a different intersection pattern makes smaller intersections with $\partial E$ since the subarcs from the outlets with different colors makes additional intersections with $\partial E$. This completes the proof.
\end{proof}

 Theorem~\ref{T0} gives a relation between two normal forms of the same rational $3$-tangle $T$ in terms of normal jump moves.

\begin{Thm}(Theorem $9$~ in \cite{3})\label{T0}
Suppose that $\mathcal{B}$ and $\mathcal{B}'$ is two normal forms of a rational $3$-tangle $T$. Then there exists a sequence of normal jump moves to lead one normal form to the other normal form.
\end{Thm}

\section{A complex of normal forms of a rational $3$-tangle}\label{s4}
\begin{figure}[htb]
\includegraphics[scale=.7]{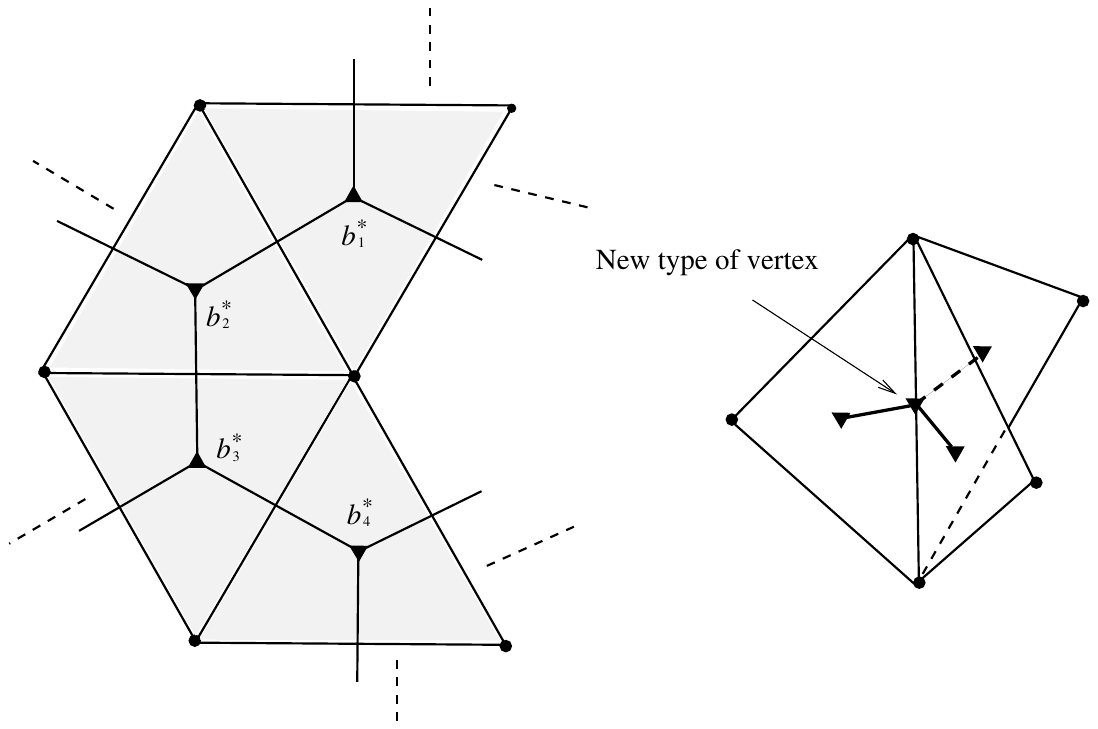}
\vskip -340pt
\caption{A part of bridge arc complex and co-bridge arc complex of $T$}\label{f2}
\end{figure}
In this section, we define a simplicial complex associate to  a rational 3-tangle.  Let $T=(B^3,\tau)$ be a rational $3$-tangle. We define the \textit{bridge arc} complex for the rational $3$-tangle $T$, denoted by $\mathcal{B}(T)$, as follows;
for vertices, we take the  isotopy classes of  bridge arcs  for $T$ on $\Sigma_{0,6}$.
 A collection of $k + 1$ vertices forms a $k$-simplex if there are representatives for each
that are pairwise disjoint.
 We note that  $\mathcal{B}(T)$  forms a $2$-complex since the number of disjoint bridge arcs in $\Sigma_{0,6}$ is at most three.  We also can define the dual-complex of $\mathcal{B}(T)$ denoted by $\mathcal{B}^*(T)$. It is called the \textit{co-bridge arc} complex for $T$. 
Each vertex of $\mathcal{B}^*(T)$ corresponds to  a collection of three isotopy classes of bridge arcs which are  realized as pairwise disjoint three bridge arcs. They are  centers of $2$-cells of $\mathcal{B}(T)$. Two vertices of them  span a $1$-simplex if the corresponding collections contain two common isotopy classes(two vertices of $\mathcal{B}(T)$).  We also note that there exists a different type of vertex which is in the middle of $1$-simplex of $\mathcal{B}(T)$ if more than two $2$-cells contain the given $1$-simplex as in the second diagram of Figure~\ref{f2}. We note that the dimension of $\mathcal{B}^*(T)$ is one since the dimension of $\mathcal{B}(T)$ is two. So, $\mathcal{B}^*(T)$ is a graph.
 Now, we consider the sub-complex of  $\mathcal{B}^*(T)$ which consists of vertices corresponding to collections of  three disjoint bridge arcs of $T$ which are normal forms and vertices in the middle of $1$-cell which is the common edge of more than two $2$-cells represent normal forms. It is called the \textit{normal complex} of $T$ denoted by $\mathcal{N}(T)$. We note that there is a path between any two vertices of them by Theorem~\ref{T0}.
So, $\mathcal{N}(T)$ is connected.  Actually, we also note that $\mathcal{B}^*(T)$ is  connected since the procedure to have a normal form is replacing one of the given three bridge arcs by a simpler bridge arc.  One of our main goals in this paper is to show it is contractible.  We will give a proof for this in next section.

\section{Contractibility of $\mathcal{N}(T)$}
\label{s5}
Recall that $E_1'$ is the barycentric two punctured disk in $E_1$ as in the diagrams of Figure~\ref{f1}. We discuss the pattern of arc components in $E_1'$ to prove the statement that the normal complex is contractible.
 First of all, we decide an order of the three simple arcs of a normal form of $T$. We label the fixed six punctures in $\Sigma_{0,6}$ from $p_1$ to $p_6$, from the left puncture in $E_1$ to the right puncture in $E_3$ in a clockwise direction as in the diagram of Figure~\ref{f3}.
 \begin{figure}[htb]
 \includegraphics[scale=.5]{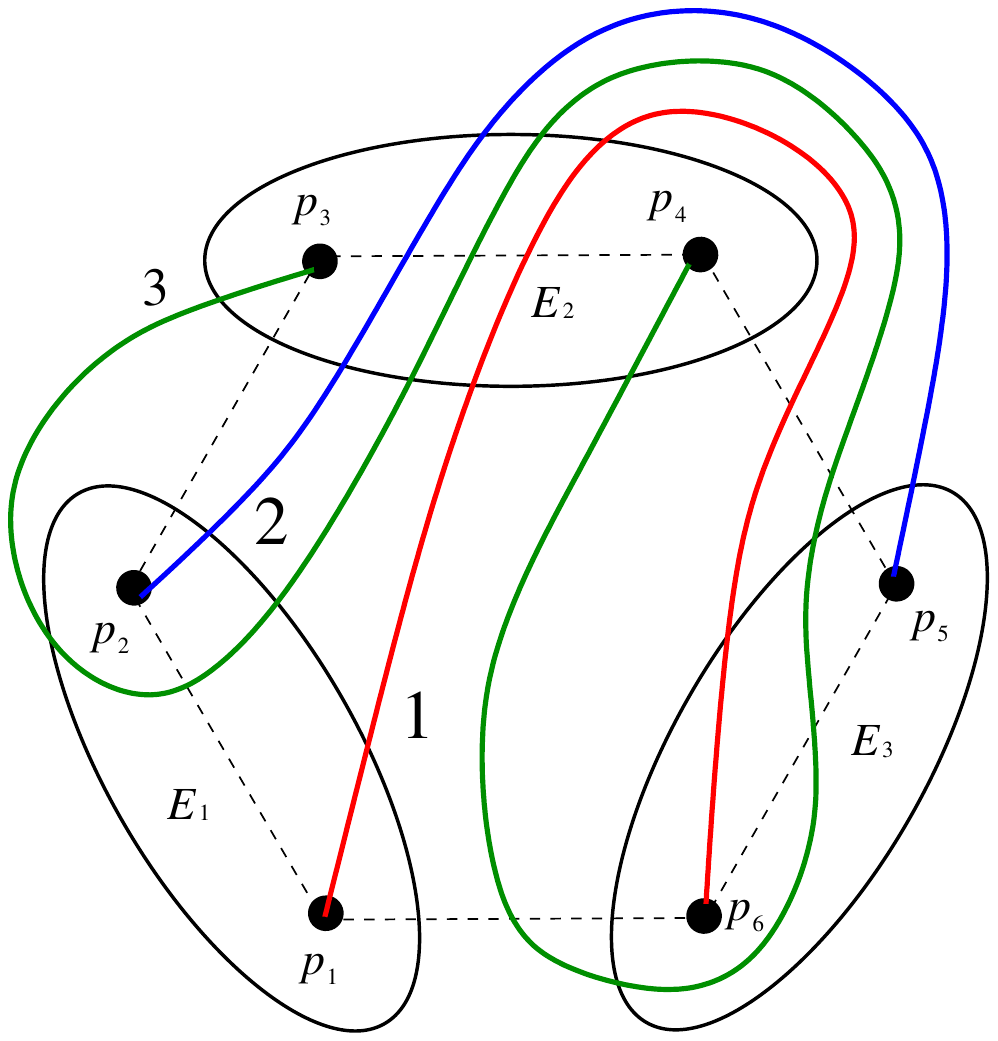}
 \vskip -160pt
 \caption{Labelling of bridge arcs}\label{f3}
 \end{figure}
Now we name the three simple arcs of a normal form of $T$ as follows. The simple arc is labelled as $1$ if one of the endpoints of it is $p_1$. If the other endpoint of the simple arc labelled as $1$ is $p_2$ then the simple arc containing $p_3$ is labelled as $2$. Otherwise, the simple arc containing $p_2$ is named as $2$. Then the last simple arc is labelled as $3$  as in the diagram of Figure~\ref{f3}.
 Let $\beta_1,\beta_2,\beta_3$ be the labelled three bridge arcs.
Let $\boldsymbol{\beta}=(\beta_1,\beta_2,\beta_3)$ be the ordered sequence of a normal form $\mathcal{B}$ of a rational $3$-tangle $T$ which is called a \textit{normal coordinate} of $T$.  Let $k_{i}^1(\boldsymbol{\beta})$ be the normal coordinate of $T$  which is obtained from $\boldsymbol{\beta}$  by standard normal jump move replacing $\beta_i$ if it exists.   Let $k_{i}^2(\boldsymbol{\beta})$ be the normal coordinate of $T$  which is obtained from $\boldsymbol{\beta}$  by another possible normal jump move replacing $\beta_i$ if it exists.   Recall that there are at most two normal jump moves replacing $\beta_i$ by Theorem~\ref{T-2}.  We call $k_{i}^1$ and $k_{i}^2$  the \textit{normal mapping}. We note that the two  bridge arcs replacing $\beta_i$ by $k_i^1$ and $k_i^2$ intersect more than two times. This implies that they meet at a point which is not  at the two endpoints. Now, we define $k_j^{\delta_j}\circ k_i^{\delta_i}(\boldsymbol{\beta})=
k_j^{\delta_j}(k_i^{\delta_i}(\boldsymbol{\beta}))$ for some $i,j\in\{1,2,3\}$, where $\delta_i, \delta_j\in\{1,2\}$. Let $(k_i^{\delta_i})^n=\overbrace{k_i^{\delta_i}
\circ k_i^{\delta_i}\circ\cdot\cdot\cdot \circ k_i^{\delta_i}}^n$. 
We note that $(k_i^{1})^2(\boldsymbol{\beta})=\boldsymbol{\beta}$ for any $i\in\{1,2,3\}$ by the construction of standard normal jump move. (Refer to~\cite{3}.)  Clearly,  we also have the following lemma since there are only three possible normal forms which preserve $\beta_j$ and $\beta_k$, where $\{i,j,k\}=\{1,2,3\}$.

  \begin{Lem}\label{L5}
 Let $h=k_i^{\delta_{i1}}
\circ k_i^{\delta_{i2}}\circ\cdot\cdot\cdot \circ k_i^{\delta_{in}}$, where $\delta_{ik}\in\{1,2\}$ for all $k\in\{1,2,3,...,n\}$. Then $h(\boldsymbol{\beta})=\boldsymbol{\beta},~ k_i^1(\boldsymbol{\beta})$ or $ k_i^2(\boldsymbol{\beta})$.  
 \end{Lem}

 For convenience, we use $k_i$ instead of $k_i^1$ or $k_i^2$ if $k_i$ means an arbitrary normal mapping of the two.
Let $|\boldsymbol{\beta}\cap \partial E_1|$ be the  geometric minimal intersection number between $\beta$ and $\partial E_1$.  We note that for a given rational $3$-tangle $T$ there exists a normal coordinate $\boldsymbol{\beta}'$ so that $|\boldsymbol{\beta}'\cap \partial E_1|\leq |k_i(\boldsymbol{\beta}')\cap \partial E_1|$ for all $i
\in\{1,2,3\}$ since $|\boldsymbol{\beta}\cap \partial E_1|<\infty$. If $\boldsymbol{\beta}'$ satisfies the condition above, we call  $\boldsymbol{\beta}'$  a \textit{minimal} normal coordinate of $T$ with respect to $\partial E_1$. The following theorem can guarantee   the global minimality for the definition of a minimal normal coordinate not just the local minimality by combining with Lemma~\ref{L5}. In other words, it is impossible to have a normal coordinate $\boldsymbol{\beta}''$ of $T$ so that
$|\boldsymbol{\beta}''\cap \partial E_1|<|\boldsymbol{\beta}'\cap \partial E_1|$ if $\boldsymbol{\beta}'$ is a minimal normal coordinate of $T$ with respect to $\partial E_1$. 

\begin{Thm}\label{T1}
Suppose that $\boldsymbol{\beta}$ is a minimal normal coordinate of a rational $3$-tangle $T$ with respect to $\partial E_1$. Let $\{k_{t_i}\}_1^n$ be a sequence of $k_{t_i}$ so that $t_i\neq t_{i+1}$ for all $\displaystyle{i\in\{1,2,...,n-1\}}$. Let $k=k_{t_n}\circ k_{t_{n-1}}\circ\cdot\cdot\cdot\circ k_{t_1}$. Then $|\boldsymbol{\beta}\cap\partial E_1|\leq |k(\boldsymbol{\beta})\cap\partial E_1|$ for any sequence of $k_{t_i}$ and $n$.
\end{Thm}

\begin{proof}

Let $\boldsymbol{\beta}=(\beta_1,\beta_2,\beta_3)$ be a minimal normal coordinate of $T$ with respect to $\partial E_1$. 
 Now, we investigate the window $\omega_1$ of $\partial E_1$ which contains all the intersections between $\beta$ and $\omega_1$ as in the diagram of Figure~\ref{c11}.
\begin{figure}[htb]
\includegraphics[scale=.6]{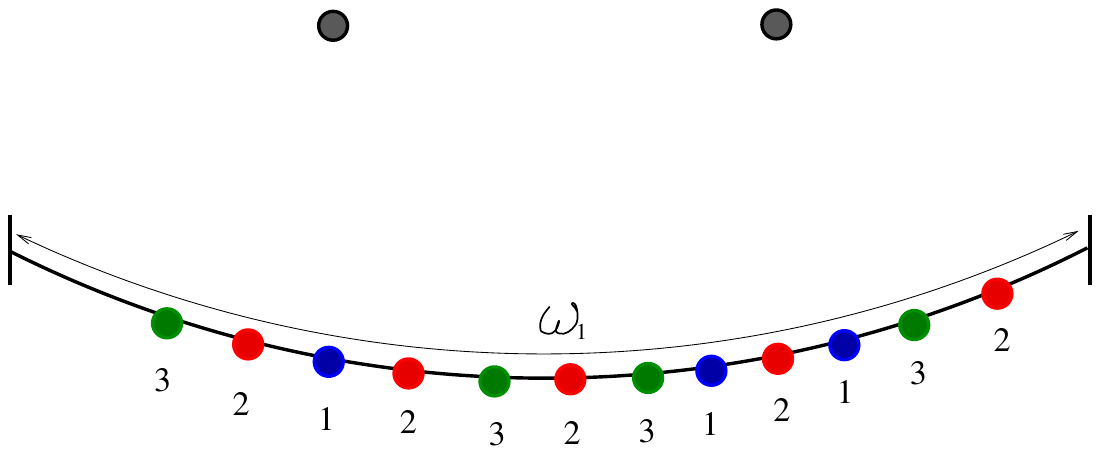}
\vskip -350pt
\caption{}\label{c11}
\end{figure}
 The blue, red and green dots stand for the intersections between $\beta_1$ and $\omega_1$, $\beta_2$ and $\omega_1$, and $\beta_3$ and $\omega_1$ respectively. Recall that we have the unique intersection pattern of $\beta_i'$ by the normal mapping $k_i$ in $E_1$ if $\beta_i'$ does not preserve the intersection pattern of $\beta_i$ in $E_1$. We consider the following cases for verifying this theorem.\\
   
 \textbf{Case 0.}  If $|\boldsymbol{\beta}\cap \partial E_1|=0$ then it is obvious that the theorem holds since $\boldsymbol{\beta}$ is the unique normal form of $\infty$ tangle. \\
  
  Now, we assume that $|\boldsymbol{\beta}\cap \partial E_1|\neq 0$.  There are two more cases  as follows.\\
 
\textbf{Case 1.}  There exists $\beta_i$ so that $\beta_i\cap\omega_1=\emptyset$ for some $i$.\\

 We note that  $\beta_j\cap\omega_1\neq\emptyset$ if $j\neq i$. Otherwise, $\beta_j\cap\omega_1=\emptyset$ for all $j\in\{1,2,3\}$ because of the normality. This is the case we already excluded.
  Without loss of generality, we assume that $\beta_3\cap\omega_1=\emptyset$ and $\beta_i\cap\omega_1\not=\emptyset$ if $i\neq 3$. 
Now, we claim that the following two statements are true. The first statement gives a proof of Theorem~\ref{T1}.
\vskip 10pt
\begin{enumerate}
\item All the original intersections between $\beta$ and $\omega_1$ would be preserved when we apply $k$ which is a composition  of normal mappings defined above.  
\item The  intersection patterns of both sides of all the intersections between $\beta$ and $\omega_1$ by $k$ are symmetric until another original intersections on both sides appear.
\end{enumerate}
\begin{proof}[Proof of the claim]
We use induction on $n$ to prove the claim, where $n$ is the number of normal mappings to construct $k$. We first prove that the claim is true if $n=1$ as follows.

 \begin{figure}[htb]
\includegraphics[scale=.6]{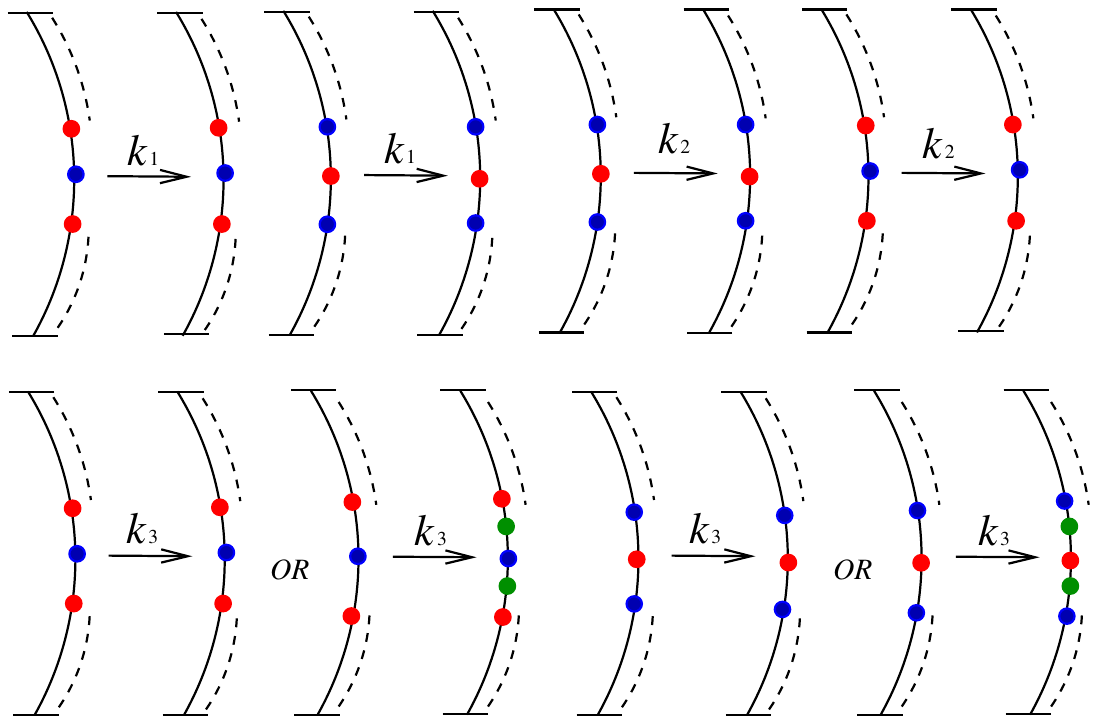}
\vskip -290pt
\caption{Case 1, $n=1$}\label{c7}
\end{figure}
  \begin{figure}[htb]
\includegraphics[scale=.7]{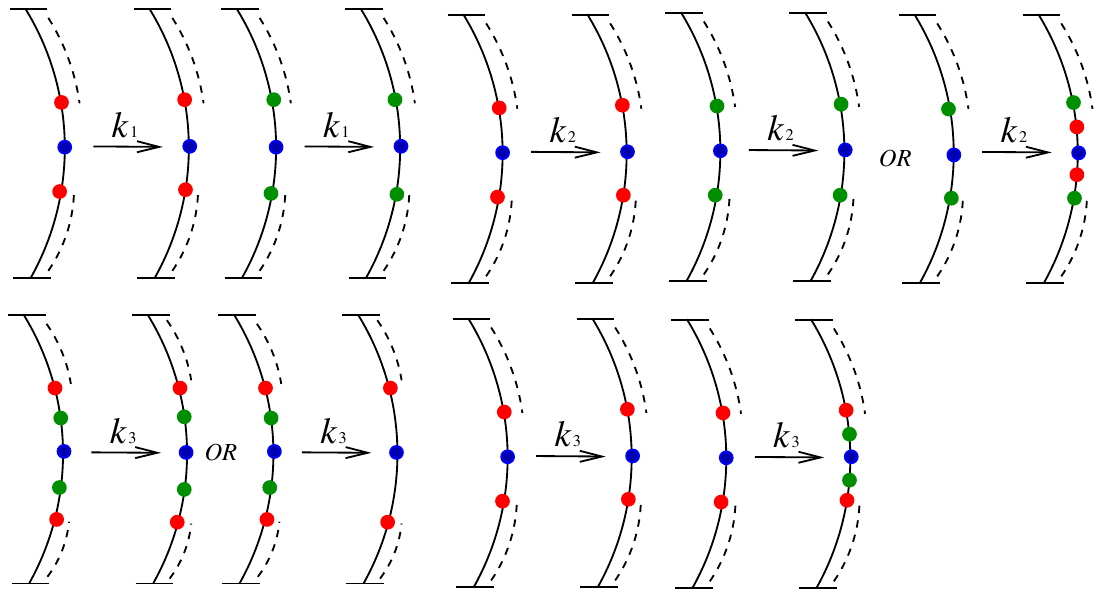}
\vskip -370pt
\caption{Case 1, $n=k$ ; in terms of a blue dot}\label{c7-2}
 \end{figure}
 
 We note that there is no change of the intersections on $\partial E_1$ if we apply $k_1$ or $k_2$ as in the upper two diagram of Figure~\ref{c7}. We note that the second condition works as well since there is no new intersections between existing intersections. If $k_3$ preserves the intersection pattern in $E$, then it is clear that the two conditions hold.  Now, suppose that $k_3$ is the normal mapping we apply and $k_3$ changes the intersection pattern in $E_1$. Then, a green dot appears in the middle of all the adjacent intersections between $\beta$ and $\omega_1$. Moreover, none of the original intersections  disappear.   This implies that the conditions $(1)$ and $(2)$ are satisfied when $n=1$.\\

Now, we assume that the two conditions $(1)$ and $(2)$ are satisfied when $n=k$. We now show that the conditions $(1)$ and $(2)$ are satisfied when $n=k+1$ as well. It is clear that the condition $(1)$ is satisfied when we apply additional normal mapping  since they satisfied the condition $(2)$ when $n=k$. Refer to Figure~\ref{c7-2}.    Moreover, the newly obtained or vanished intersections by the additional normal mapping are symmetric since the given patterns are symmetric based on the original intersections. This implies that the   right and left side of the original intersections are still symmetric after applying the additional normal mapping. Therefore, the claim works when $n=k+1$. This completes the proof of  the claim.

 \end{proof}

\textbf{Case 2.}  $\beta_i\cap \partial E_1\neq \emptyset$ for all $i\in\{1,2,3,\}$.\\

We also use induction on $n$ to cover this case. Recall that there are at most two normal mappings for $k_i$. However, there is unique intersecting pattern with $\partial E$ if the intersection pattern is different from the intersection pattern of $\beta$ with $\partial E$.
 It is clear that   $|\boldsymbol{\beta}\cap\partial E_1|\leq |k_{t_1}(\boldsymbol{\beta})\cap\partial E_1|$ since $\boldsymbol{\beta}$ is a minimal normal coordinate of $T$ with respect to $\partial E_1$. Now, we assume that $|\boldsymbol{\beta}\cap\partial E_1|\leq |k(\boldsymbol{\beta})\cap\partial E_1|$, where $k=k_{t_n}\circ k_{t_{n-1}}\circ\cdot\cdot\cdot\circ k_{t_1}$ such that $t_i\neq t_{i+1}$ for all $i\in\{1,2,...,n-1\}$ for $n\geq 2$. Let $k(\boldsymbol{\beta})=(\beta_1',\beta_2',\beta_3')$. 
Without loss of generality, we
 assume that $t_n=1$ and $t_{n+1}=2$ since it is a matter of the order of $\boldsymbol{\beta}$. 
 Let $(\beta_1'',\beta_2',\beta_3')=k_1^{-1}\circ k(\boldsymbol{\beta})=k_{t_{n-1}}\circ k_{t_{n-2}}\circ\cdot\cdot\cdot\circ k_{t_1}(\boldsymbol{\beta})$ and $(\beta_1',\beta_2'',\beta_3')=k_2\circ k(\boldsymbol{\beta})$. We use  blue, red and green colors for $\beta_1',\beta_2''$ and $\beta_3'$  respectively in the diagrams of Figure~\ref{c21}.
  We  claim that $|\beta_2'\cap \omega_1|\leq|\beta_2''\cap \omega_1|$.  We note that the claim implies that $|\boldsymbol{\beta}\cap\partial E_1|\leq |k(\boldsymbol{\beta})\cap\partial E_1|\leq |k_2\circ k(\boldsymbol{\beta})\cap\partial E_1|$ and this completes the induction. If the intersection patterns of  $ k(\boldsymbol{\beta})$ and $k_2\circ k(\boldsymbol{\beta})$  in $E_1$ are the same, then the claim is obvious. So, we assume that $ k(\boldsymbol{\beta})$ and $k_2\circ k(\boldsymbol{\beta})$  have different intersection patterns with $E_1$. Also, we assume that $k_1^{-1}\circ k(\boldsymbol{\beta})$ and $k(\boldsymbol{\beta})$  have different intersection patterns in $E_1$ for the argument below to prove the claim.
  If they have the same intersection pattern then we consider the previous normal coordinate in terms of the composition instead of  $k_1^{-1}\circ k(\boldsymbol{\beta})$  which is the first normal coordinate having a different intersection pattern with $E_1$ when we count the sequence of normal mappings for $k$ backward. If there is no such normal coordinate then the claim is clear since $\boldsymbol{\beta}$ is a minimal normal coordinate.
  \begin{proof}[Proof of the claim]
  
  The  dots on the first horizontal segment of each diagram of Figure~\ref{c21} depict the part of the intersections between $\omega_1$ and $\{\beta_1'',\beta_2',\beta_3'\}$. The dots on the third horizontal segment of each diagram of Figure~\ref{c21} present the part of the intersections between $\omega_1$ and $\{\beta_1',\beta_2'',\beta_3'\}$. If the number of intersections between $\beta_2'$ and $\omega_1$ is 1, then it is clear that $|\beta_2'\cap \omega_1|\leq|\beta_2''\cap \omega_1|$ since the intersection should be one of the endpoints of the arc connecting one of the punctures in $E_1$. We note that the arc is never vanished by a normal jump move if $|\boldsymbol{\beta}\cap\partial E_1|\neq \emptyset$. Now, we assume that $|\beta_2'\cap \omega_1|>1$. We investigate three types of two or three consecutive dots from a red dot as follows.\vskip 10pt
    \begin{figure}[htb]
  \includegraphics[scale=.87]{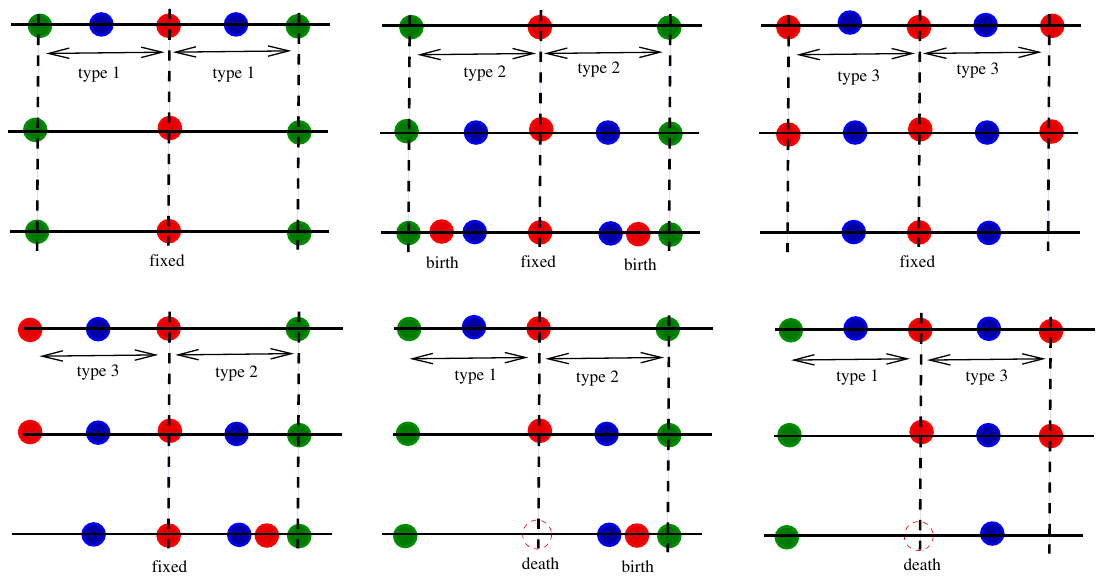}
  \vskip -470pt
  \caption{Changes of intersections by $k_2\circ k_1$}\label{c21}
  \end{figure}

  $\bullet$ Type 1 : the order of colors for the consecutive dots is red, blue and green from the left or from the right. 
  \vskip 10pt
 $\bullet$ Type 2 : the order of colors for the consecutive dots is red and green from the left or from the right. 
  \vskip 10pt
   $\bullet$ Type 3 : the order of colors for the consecutive dots is red, blue and red from the left or from the right. 
  
 \vskip 10pt

 The diagrams of Figure~\ref{c21} illustrate all the possible cases having the three types when we investigate the patterns of intersections between two consecutive red dots up to symmetry. The vertical dotted lines indicate some important corresponding dots from $(\beta_1'',\beta_2',\beta_3')$ to $(\beta_1',\beta_2',\beta_3')$ and  from $(\beta_1',\beta_2',\beta_3')$ to $(\beta_1',\beta_2'',\beta_3')$ when we apply $k_2\circ k_1$.\\

   First of all, we assume that a red dot is the common dot of two same types(type 1, type 2 or type 3) in $(k_1^{-1}\circ k(\boldsymbol{\beta}))\cap\omega_1$. We can check that the red dot is preserved for each case as in the upper diagrams of Figure~\ref{c21}.
   Now, we assume that a red dot is the common  dot of two different types as in the lower diagrams of Figure~\ref{c21}. We note that if the two types having the common red dot are type 2 and type 3 then the red dot is preserved as well. However, if the two types are type 1 and type 2, or type 1 and type 3 then the red dot would be vanished.
     We focus on the fact that the common type of the last two cases mentioned above is type 1. We note that the blue dot of the type 1 is vanished when we apply $k_1$.  We claim that there exist at least $n$ type 2 if there are $n$ type 1.   By the assumption that $|\boldsymbol{\beta}\cap\partial E_1|\leq |k(\boldsymbol{\beta})\cap\partial E_1|$, there are at least $n$ newborn blue dots if the number of vanished  blue dots  because of type 1 is $n$. We note that if the colors of the adjacent both sides of a blue  dot are the same then the blue dot would be preserved. The two patterns of type 1 are the only two cases that the blue dot is vanished when we apply $k_1$. Because of the normality, it is impossible to have  consecutive dots with the same color. Therefore, we should have at least $n$ type 2 in $|k_1\circ k(\boldsymbol{\beta})\cap \partial E_1|$.
      The diagrams of Figure~\ref{c21} contains all the cases which can have a type 2. Moreover, there is newborn red dot in the bottom row between the original green and red dots. So, we can make an injective mapping from  the red dots of  $|k_1\circ k(\boldsymbol{\beta})\cap \partial E_1|$ to the red dots of   $|k_2\circ k(\boldsymbol{\beta})\cap \partial E_1|$.  This implies that $|\beta_2'\cap \omega_1|\leq |\beta_2''\cap \omega_1|$ and it completes the claim.
  \end{proof}
 By Case $0$, Case $1$ and Case $2$, we completes the prove of Theorem~\ref{T1}.
\end{proof}

\begin{Cor}\label{C2}
Suppose that $\boldsymbol{\beta}$ is a minimal normal coordinate of a rational $3$-tangle $T$ with respect to $\partial E_1$ so that $|\boldsymbol{\beta}\cap\partial E_1|<|k_i(\boldsymbol{\beta})\cap \partial E_1|$ for all $i$. Then the minimal normal coordinate $\boldsymbol{\beta}$ of $T$ with respect to $\partial E_1$ is unique.
\end{Cor}
The proof of Corollary~\ref{C2} is obvious by Lemma~\ref{L5} and the argument in Theorem~\ref{T1}. If a minimal normal coordinate $\boldsymbol{\beta}$ of a rational $3$-tangle $T$ with respect to $\partial E_1$ satisfies $|\boldsymbol{\beta}\cap\partial E_1|<|k_i(\boldsymbol{\beta})\cap \partial E_1|$ for all $i$ then it is called the \textit{strict} minimal normal coordinate of $T$ with respect to $\partial E_1$.\\

Recall that if a minimal normal coordinate $\boldsymbol{\beta}$ stands for the $\infty$ tangle then it is unique normal form. So, there is no normal mapping $k_i$ for any $i$. If  one of the bridge arcs of a minimal normal coordinate $\boldsymbol{\beta}$ not presenting the $\infty$ tangle is contained in $E_i$, then some of $k_j$ cannot exist. In this case, we have $|\boldsymbol{\beta}\cap\partial E_1|=|k_i(\boldsymbol{\beta})\cap \partial E_1|$ for the existing $k_j$. Moreover, it is impossible to have $|\boldsymbol{\beta}\cap\partial E_1|=|k_i(\boldsymbol{\beta})\cap \partial E_1|$ for all $i\in\{1,2,3\}$ if there exist normal mapping $k_1,k_2$ and $k_3$ by the  lemma below. 

\begin{Lem}\label{T3}
Let $\boldsymbol{\beta}=(\beta_1,\beta_2,\beta_3)$ be a minimal normal coordinate of a rational $3$-tangle $T$ with respect to $\partial E_1$. Suppose that $k_1, k_2$ and $k_3$  exist. Then there exists $k_i$ so that $|\boldsymbol{\beta}\cap\partial E_1|<|k_i(\boldsymbol{\beta})\cap \partial E_1|$.
\end{Lem}

\begin{proof}
For a contradiction for this lemma, we assume that $|\boldsymbol{\beta}\cap\partial E_1|=|k_i(\boldsymbol{\beta})\cap \partial E_1|$ for all $i\in\{1,2,3\}$. 
\begin{figure}[htb]
\includegraphics[scale=.3]{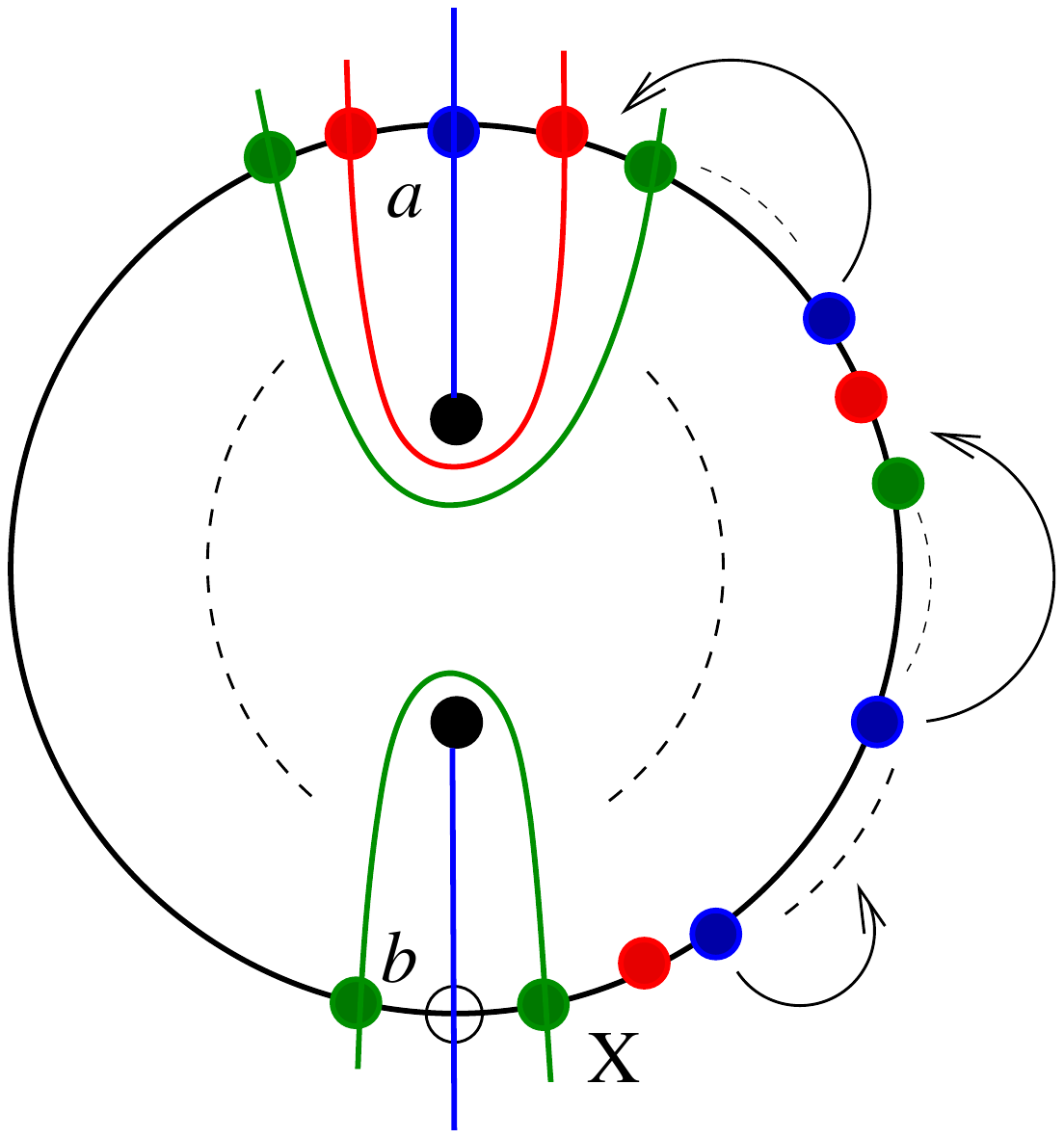}
\vskip -70pt
\caption{A schematic diagram for the intersections between $\boldsymbol{\beta}$ and $\partial E_1$}\label{c8}
\end{figure}
We can exclude the case that $\boldsymbol{\beta}\cap\partial E_1=\emptyset$ because of the previous comment. \\

Now, we assume that there exists $\beta_i$ so that $\beta_i\cap \partial E_1=\emptyset$. Then, we note that  $\beta_j\cap \partial E_1\neq\emptyset$ if $i\neq j$ since  $\boldsymbol{\beta}$ is a normal coordinate. Then it is clear that $|\boldsymbol{\beta}\cap\partial E_1|<|k_i(\boldsymbol{\beta})\cap \partial E_1|$ since there exists $k_i$ by the assumption. It is possible to have $|\boldsymbol{\beta}\cap\partial E_1|=|k_i(\boldsymbol{\beta})\cap \partial E_1|$ if $\beta_i\subset E_k$ for some $k$. However, there is no $k_i$ in this case.\\

 Now, we assume that $\beta_i\cap\partial E_1\neq \emptyset$ for all $i$. Then we note that there exist $k_1^1, k_2^1$ and $k_3^1$ which are standard normal mappings. Let $k_1=k_1^1, k_2=k_2^1$ and $k_3=k_3^1$. We note  that then $k_1, k_2$ and $k_3$ do not preserve the intersection pattern in $E_1$.
Let $a$ and $b$ be the two intersections between $\beta(=\beta_1\cup\beta_2\cup\beta_3)$ and $\partial E_1$ which are the two endpoints of the arc components  connecting the two punctures in $E_1$ directly. We note that the pattern of the dots in the two half circles between $a$ and $b$ are symmetric as in the diagram of Figure~\ref{c8}. We want to point out that the diagram is a schematic diagram to explain the situation. In the diagram, we use  blue, red and green colors to distinguish the dots coming from the three different $\beta_i$. Without loss of generality, $a$ belongs to $\beta_1$(blue). We give the orders of dots with the same color from upper dot close to $a$ by following the right half circle continuously. We investigate the ordered sequence of blue dots. We claim that if all of the consecutive ordered two blue dots from $a$ contain at least two dots then it is impossible to have $|\boldsymbol{\beta}\cap\partial E_1|=|k_1(\boldsymbol{\beta})\cap \partial E_1|$. We note that there exists a corresponding position for each blue dot   when we consider $k_1(\boldsymbol{\beta})$ as in the diagram above. Moreover, under the condition above, there is at least one unoccupied site named $X$ as in the diagram of Figure~\ref{c8}.
 \begin{figure}[htb]
\includegraphics[scale=.7]{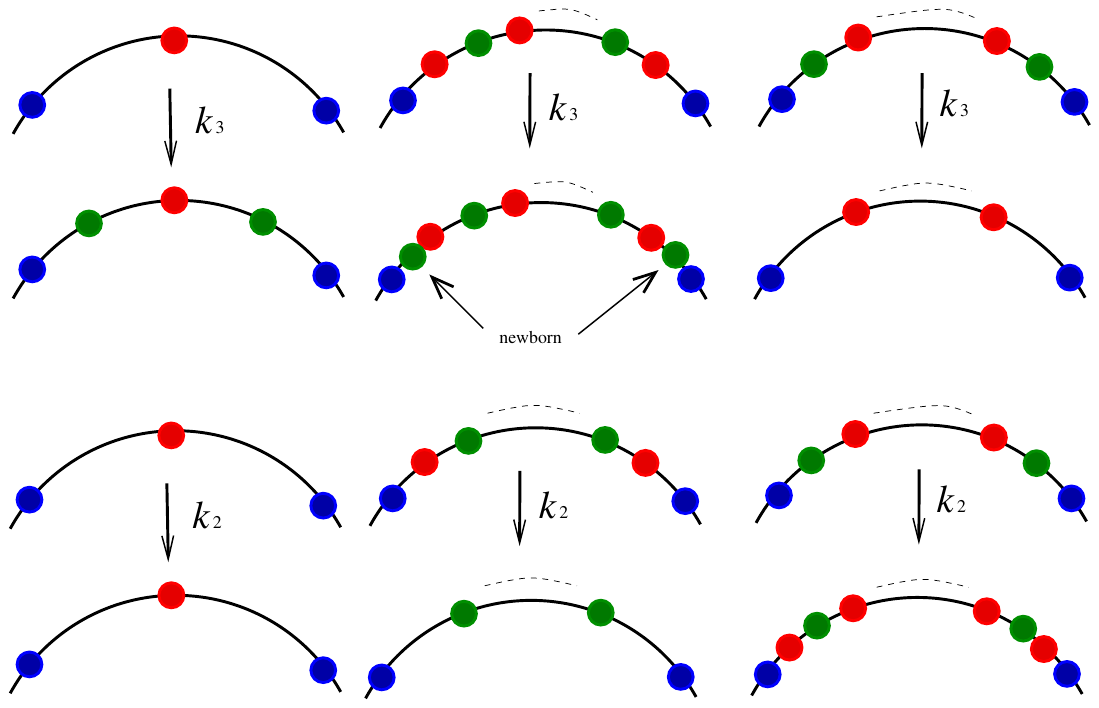}
\vskip -350pt
\caption{Changes of types}\label{c22}
 
\end{figure}
 We also note that $a$ and $b$ would be preserved since the  adjacent both sides have the intersections with same color. So, there exists at least one pair of ordered two blue dots which contains a single dot with a different color.  Suppose that the only dot between the two blue dots is a red dot which belongs to $\beta_2$ without loss of generality. It is called a \textbf{BRB} type. We have two newborn green dots in the middles between the blue and the red dot in this case when we have $k_3(\boldsymbol{\beta})$. We also can define  a \textbf{BGB} type.\\
 
   We claim that there exists a pair of consecutive  blue dots which contains odd number of dots between them and the first and the last dots are green if it satisfies $|\boldsymbol{\beta}\cap\partial E_1|=|k_3(\boldsymbol{\beta})\cap \partial E_1|$. 
We note that if a pair of consecutive blue dots contains even number of dots between them the number of green dots is preserved when we have $k_3(\boldsymbol{\beta})$. Therefore, there exists  a pair of consecutive blue dots which contains odd number of dots since $\mathbf{BRB}$ types make newborn green dots. Moreover, if the first and the last dots are red then we have two newborn green dots when we apply $k_3$ to $\boldsymbol{\beta}$ as in the middle diagram of Figure~\ref{c8}. This completes the proof of the claim.\\

 The type of the claim is called a \textbf{BG-GB} type. Similarly, we can define a \textbf{BR-RB} type.  We note that  the number of \textbf{BRB} and \textbf{BR-RB} types is the same with the number of \textbf{BG-GB} type if $\boldsymbol{\beta}$ satisfies the condition that $|\boldsymbol{\beta}\cap\partial E_1|=|k_3(\boldsymbol{\beta})\cap \partial E_1|$.
However,  this makes a contradiction since it is impossible to have $|\boldsymbol{\beta}\cap\partial E_1|=|k_2(\boldsymbol{\beta})\cap \partial E_1|$.  We note that \textbf{BRB} types do
 not make any newborn red dot when we apply $k_2$. So, the number of \textbf{BR-RB} type is the same with the number of \textbf{BGB} and \textbf{BG-GB} types. However, it is impossible since the number of \textbf{BRB} type is not zero. This completes the poof of this lemma.

\end{proof}
 
Recall that $\boldsymbol{\beta}=(\beta_1,\beta_2,\beta_3)$ is a minimal normal coordinate of a rational 3-tangle $T$ with respect to $\partial E_1$ and $\beta=\beta_1\cup\beta_2\cup\beta_3$. We note that there is no $k_i^2$ which preserves the intersection pattern of $\beta$ with $\partial E_1$ if $|\boldsymbol{\beta}\cap\partial E_1|<|k_i^1(\boldsymbol{\beta})\cap \partial E_1|$ because of Lemma~\ref{L1}.
Now, we investigate the cases that  $|\boldsymbol{\beta}\cap\partial E_1|=|k_i(\boldsymbol{\beta})\cap \partial E_1|$ for some $i$. By Lemma~\ref{T3}, it is impossible to have $|\boldsymbol{\beta}\cap\partial E_1|=|k_i(\boldsymbol{\beta})\cap \partial E_1|$ for all $i$.

\begin{Thm}\label{T4} Let $\boldsymbol{\beta}=(\beta_1,\beta_2,\beta_3)$ be a minimal normal coordinate of a rational 3-tangle $T$ with respect to $\partial E_1$ so that $\beta_i\cap\partial E_1\neq \emptyset$ for all $i\in\{1,2,3\}$.
Suppose that $|\boldsymbol{\beta}\cap\partial E_1|=|k_1(\boldsymbol{\beta})\cap \partial E_1|$.  Then  $|\boldsymbol{\beta}\cap\partial E_1|<|k_2\circ k_1(\boldsymbol{\beta})\cap \partial E_1|$ or  $|\boldsymbol{\beta}\cap\partial E_1|<|k_3\circ k_1(\boldsymbol{\beta})\cap \partial E_1|$. Moreover, if $|\boldsymbol{\beta}\cap\partial E_1|<|k_i(\boldsymbol{\beta})\cap \partial E_1|$ for $i=2,3$ and $|\boldsymbol{\beta}\cap\partial E_1|<|k_j\circ k_1(\boldsymbol{\beta})\cap \partial E_1|$ for $j=2,3$ then there are only two minimal normal coordinates of $T$ with respect to $\partial E_1$.
\end{Thm}
\begin{proof}
We note that we have $k_i^1$ for all $i\in\{1,2,3\}$ if  $\beta_j\cap\partial E_1\neq \emptyset$ for all $j\in\{1,2,3\}$. We also note that there is no $k_i^2$ by Lemma~\ref{L1}. By Lemma~\ref{T3}, it is clear that $|\boldsymbol{\beta}\cap\partial E_1|<|k_2\circ k_1(\boldsymbol{\beta})\cap \partial E_1|$ or  $|\boldsymbol{\beta}\cap\partial E_1|<|k_3\circ k_1(\boldsymbol{\beta})\cap \partial E_1|$ since $k_1(\boldsymbol{\beta})$ is a minimal normal coordinate as well by the assumption and $(k_1)^2(\boldsymbol{\beta})=\boldsymbol{\beta}$. Moreover, if  $|\boldsymbol{\beta}\cap\partial E_1|<|k_i(\boldsymbol{\beta})\cap \partial E_1|$ for $i=2,3$ and $|\boldsymbol{\beta}\cap\partial E_1|<|k_j\circ k_1(\boldsymbol{\beta})\cap \partial E_1|$ for $j=2,3$ then $\boldsymbol{\beta}$ and  $k_1^1(\boldsymbol{\beta})$ are only two minimal normal coordinates of $T$ with respect to $\partial E_1$. 
\end{proof}

The following two theorems cover the rest of cases for the number of minimal normal coordinates of $T$. For an easier argument, consider the consecutive intersections between $\beta_3$ and $\partial E_1$ in $\partial E_1$. Recall  that \textbf{GR-RG}, \textbf{GRG}, \textbf{GB-BG}, and \textbf{GBG} types. Let $n_*$ be the number of $*$ type included among $ \boldsymbol{\beta}\cap \partial E_1$ in $\partial E_1$. \\

\begin{Thm}\label{T5}
Let $\boldsymbol{\beta}=(\beta_1,\beta_2,\beta_3)$ be a minimal normal coordinate of a rational 3-tangle $T$ with respect to $\partial E_1$ such that $\beta_i\cap \partial E_1\neq \emptyset$ for all $i=1,2,3$ and  $|\boldsymbol{\beta}\cap\partial E_1|=|k_j
(\boldsymbol{\beta})\cap \partial E_1|$ for $j=1,2$. Then $\boldsymbol{\beta}$ satisfies $n_{\textbf{GRG}}=n_{\textbf{GBG}}=0$ and $n_{\textbf{GR-RG}}=n_{\textbf{GB-BG}}$. Moreover, it  satisfies the following condition.
\begin{enumerate}
\item If $n_{\textbf{\textbf{GR-RG}}}=0$ then there are only two minimal normal coordinates of $T$ with respect to $\partial E_1$.

\item If $n_{\textbf{GR-RG}}\neq 0$, the number of minimal normal coordinates of $T$ with respect to $\partial E_1$ is ${1\over 2}\left(\xi_1+\xi_2\right)$, where $\xi_1$ and $\xi_1$ are the minimal  intersection numbers among the intersection numbers between the  pairs of consecutive green dots of $\textbf{GR-RG}$ type and  $\textbf{GB-BG}$ type  respectively.
\end{enumerate}

\end{Thm}

\begin{proof}
\begin{figure}[htb]
\includegraphics[scale=.9]{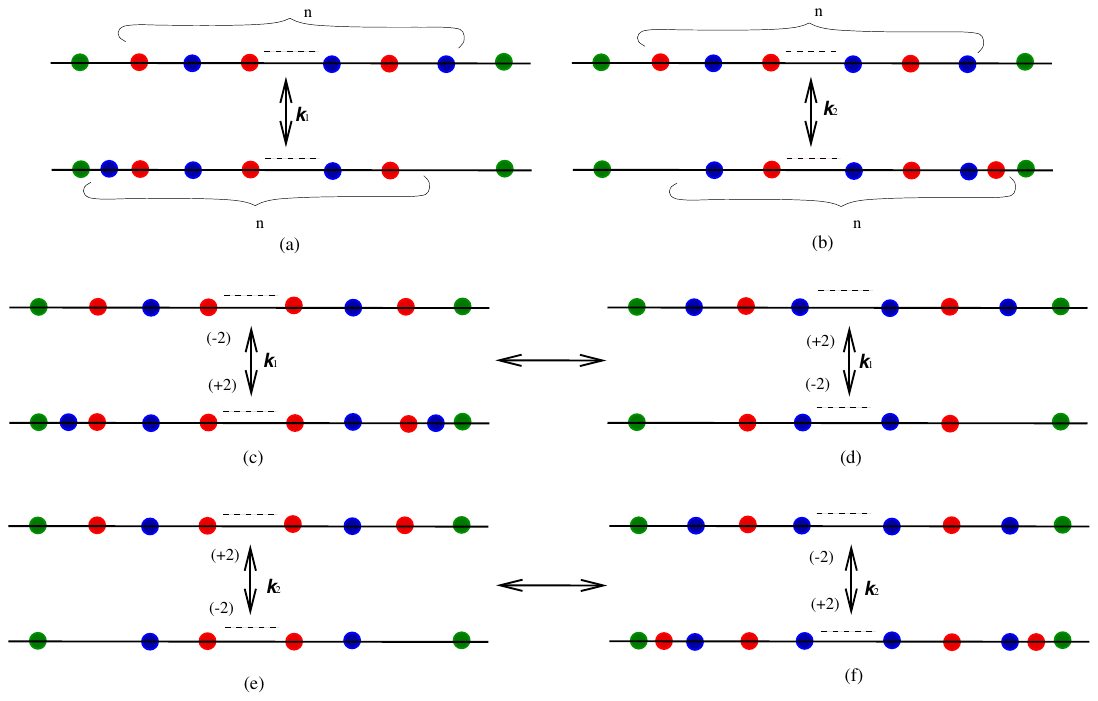}
\vskip -440pt
\caption{}\label{c28}
\end{figure}

Since $\beta_i\cap \partial E_1\neq \emptyset$ for all $i=1,2,3$, there exists the standard normal jump move $k_{i}^1$ for all $i\in\{1,2,3\}$. However, there is no $k_{i}^2$  by Lemma~\ref{L1}. So, we let $k_i=k_i^1$.
We note that if a pairs of consecutive intersections between $\beta_3$ and $\partial E_1$ in $\partial E_1$ contains  even number of intersections then the normal jump move by $k_1$ or $k_2$ would preserve the number of intersections between them as in the diagrams $(a)$ and $(b)$ of Figure~\ref{c28}.  This implies that $n_{\textbf{GR-RG}}+n_{\textbf{GRG}}=n_{\textbf{GB-BG}}$ and  $n_{\textbf{GR-RG}}=n_{\textbf{GB-BG}}+n_{\textbf{GBG}}$  by applying $k_1$ and $k_2$ respectively since $|\boldsymbol{\beta}\cap\partial E_1|=|k_i(\boldsymbol{\beta})\cap \partial E_1|$ for $i=1,2$. (Refer to Figure~\ref{c28} to see the increase or the decrease of intersections of each case.) Therefore, $n_{\textbf{GRG}}=n_{\textbf{GBG}}=0$ and $n_{\textbf{GR-RG}}=n_{\textbf{GB-BG}}$. \\

We  first assume that $n_{\textbf{GR-RG}}=0$. We recall that $(k_1)^2(\boldsymbol{\beta})=(k_2)^2(\boldsymbol{\beta})=\boldsymbol{\beta}$. Moreover, we note that the color patterns between the pair of consecutive intersections between $\beta_3$ and $\partial E_1$ in $\partial E_1$ containing even number of intersections  are the same when we apply  $k_1\circ k_2$ or $k_2\circ k_1$.  We also note that we cannot have a minimal normal coordinate when we apply $k_3$ by Lemma~\ref{T3}. So, we cannot have any other minimal coordinate of $T$. This implies that there are only two minimal normal coordinates of $T$ with respect to $\partial E_1$. \\

 Now, we assume that $n_{\textbf{GR-RG}}\neq 0$.  We  apply $k_1$ and $k_2$ alternatively to find and count the number of minimal normal coordinates of $T$ with respect to $\partial E_1$ until the new obtained minimal normal coordinate of $T$ still satisfies the assumption of this theorem since $ k_i^2(\boldsymbol{\beta})\cap \partial E_1=\boldsymbol{\beta}\cap \partial E_1$ and   $|\boldsymbol{\beta}\cap \partial E_1| < |k_3(\boldsymbol{\beta})\cap \partial E_1|$ and Theorem~\ref{T1}.
Let $\xi_1$ be the minimal  intersection numbers among the intersection numbers between the  pairs of the green intersections of \textbf{GR-RG} type.
Also, let $\xi_2$ be the minimal  intersection numbers among the interection numbers between the  pairs of the green intersections of \textbf{GB-BG} type.\\

 We note that if we apply $k_1$ first, then the intersection numbers between the pairs of \textbf{GR-RG} type are increasing by $2\times n$, where $n$ is the number of $k_1$ and $k_2$ we apply. On the other hand, the intersection numbers between the pairs of  \textbf{GB-BG} type are decreasing by $2\times n$ if $2\times n<\xi_1$. 
Similarly, we we apply $k_2$ first  the intersection numbers between the pairs  of  \textbf{GB-BG} type are increasing by $2\times n$  if $2\times n<\xi_2$ and the intersection numbers between the pairs of  \textbf{GR-RG} type are decreasing by $2\times n$ if $2\times n<\xi_2$.  We note that $\xi_1$ and $\xi_2$ are odd number. Therefore,  $1+{\xi_1-1\over 2}+ {\xi_2-1\over 2}={\xi_1+\xi_2\over 2}$ is the number of minimal coordinates of $T$ with respect to $\partial E_1$. We need to consider the case that the last minimal coordinate of $T$ (when we follow this process) is preserved the normality when we apply $k_3$. However, this case cannot happen. We note that the both sides of each green dot would be changed at the same time when we apply $k_1$ or $k_2$ since there is no \textbf{GBG} and \textbf{GRG} types until we have the last minimal coordinate of $T$. Moreover, the total number of intersections between consecutive green dots of \textbf{GB-BG},\textbf{GR-RG},\textbf{GBG} and \textbf{GRG}  types for the last minimal normal coordinate of $T$ and the previous one are the same. This implies that the number of intersections between the new coordinate and $\partial E_1$ would increase if we apply $k_3$ to the last one. 
\end{proof}

\begin{Thm}\label{T5-1}
Let $\boldsymbol{\beta}=(\beta_1,\beta_2,\beta_3)$ be a minimal normal coordinate of a rational $3$-tangle $T$ with respect to $\partial E_1$ such that $\beta_i\cap \partial E_1\neq \emptyset$ for all $i=1,2,3$ and  $|\boldsymbol{\beta}\cap\partial E_1|=|k_1
(\boldsymbol{\beta})\cap \partial E_1|$  and  $|\boldsymbol{\beta}\cap\partial E_1|<|k_j
(\boldsymbol{\beta})\cap \partial E_1|$ for $j=2,3$. Then $\boldsymbol{\beta}$ has a $\textbf{GRG}$ type in $\partial E_1$ and the number of minimal normal coordinates of $T$ with respect to $\partial E_1$ is ${1\over 2}(\xi_1+1)$, where $\xi_1$ is the minimal intersection number among the intersection numbers between the pairs of consecutive green dots of $\textbf{GB-BG}$ type.
\end{Thm}
\begin{proof}
 Recall that $k_i=k_i^1$ by the argument in the previous theorem. We first claim that $\boldsymbol{\beta}$  has a $\textbf{GRG}$  type in $\partial E_1$. If there is no $\textbf{GBG}$, $\textbf{GRG}$ , $\textbf{GB-BG}$  or $\textbf{GR-RG}$ types, then we note that  $|\boldsymbol{\beta}\cap\partial E_1|=|k_2
(\boldsymbol{\beta})\cap \partial E_1|$. This violates the assumption. Since  $|\boldsymbol{\beta}\cap\partial E_1|=|k_1
(\boldsymbol{\beta})\cap \partial E_1|$,  we have $n_{\textbf{GR-RG}}+n_{\textbf{GRG}}=n_{\textbf{GB-BG}}+n_{\textbf{GBG}}$. If there is no $\textbf{GBG}$ or  $\textbf{GRG}$ type then   $|\boldsymbol{\beta}\cap\partial E_1|=|k_2
(\boldsymbol{\beta})\cap \partial E_1|$. This violates the assumption of this theorem. We note that $\boldsymbol{\beta}$ canont have $\textbf{GBG}$ type since  $|\boldsymbol{\beta}\cap\partial E_1|=|k_1
(\boldsymbol{\beta})\cap \partial E_1|$. Therefore, $\boldsymbol{\beta}$  should have $\textbf{GRG}$ type. Moreover, if we continue to appy $k_1$ and $k_2$ alternatively to $\boldsymbol{\beta}$  until the modified coordinate has a $\textbf{GBG}$ type, we still  have a minimal normal coordinate of $T$ with respect to $\partial E_1$. We note that the intersection number between a pair of consecutive green dots of $\textbf{GB-BG}$ type during each process above is decreased by $2$. Therefore, the number of minimal coordinates of $T$ with respect to $\partial E_1$ is ${1\over 2}(\xi_1+1)$. This completes the proof of the theorem.

\end{proof}
\begin{Lem}\label{T6}
Suppose that $\boldsymbol{\beta}=(\beta_1,\beta_2,\beta_3)$ is a normal coordinate of a rational $3$-tangle $T$ such that $\beta_3\cap \partial E_1=\emptyset$ and $\beta_i\cap \partial E_1\neq \emptyset$ if $i=1,2$. Let $(p_1, q_1, p_2, q_2, p_3, q_3)$ be the Dehn's parametrization of $\beta(=\beta_1\cup\beta_2\cup\beta_3)$. Then the normal mapping $k_1$ or $k_2$ to $\boldsymbol{\beta}$ preserves the number of intersections between $\beta_i$ and $\partial E_1$ for $i=1,2$. Moreover, $k_1$ and $k_2$ change the connecting pattern $q_1$ in $E_1$ into $q_1+1$ and $q_1-1$ respectively.
\end{Lem}
\begin{proof}
\begin{figure}[htb]
\includegraphics[scale=.80]{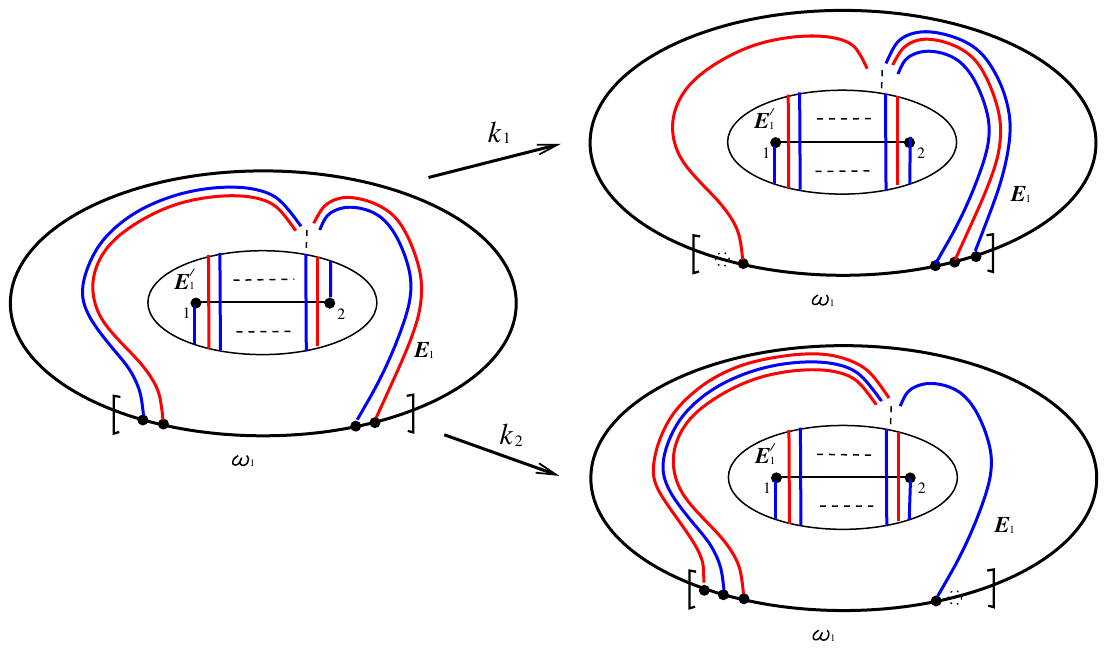}
\vskip -420pt
\caption{Connecting pattern $q_1$ changes by $k_1$ and $k_2$}\label{c37}
\end{figure}

Applying the normal mapping $k_1$ to $\boldsymbol{\beta}$ changes $\beta_1$ to $\beta_1'$ so that $(\beta_1',\beta_2,\beta_3)$ is a normal coordinate of $T$. We note that there exists the standard normal jump move $k_1^1$. We note that $\beta_1'$ is unique by $k_1^1$ since  $(\beta_1,\beta_2,\beta_3)$ is a minimal coordinate of $T$.  In other words, there is no $\beta_1''$ so that $\beta_1''$ is not isotopic to $\beta_1'$ or $\beta_1$ and  $(\beta_1'',\beta_2,\beta_3)$ is a normal form. Since $k_1$ preserves $\beta_2$, there is only way to have $\beta_1'$ as in the right upper diagram of Figure~\ref{c37}. We note that $k_1$ preserves $p_1$ but it increases $q_1$ by $1$. Similarly, the normal mapping $k_2$ changes $\beta_2$ to $\beta_2'$ and $\beta_2'$ is unique. Moreover, $k_2$ preserves $p_1$ but it decreases $q_1$ by $1$ as in the right lower diagram of Figure~\ref{c37}. This completes the proof.
\end{proof}

\begin{Lem}\label{T6-1}
Suppose that $\boldsymbol{\beta}=(\beta_1,\beta_2,\beta_3)$ is a normal coordinate of a rational $3$-tangle $T$ such that $\boldsymbol{\beta}\cap \partial E_1=\emptyset$. i.e., $\beta_1\subset E_1$. Let $(0,0,p_2, q_2, p_3, q_3)$ be the Dehn's parametrization of $\beta$. Then $p_2=p_3$ and $q_2+q_3=c$, where $c$ is a fixed integer depending on $T$.
\end{Lem}
\begin{figure}[htb]
\includegraphics[scale=.78]{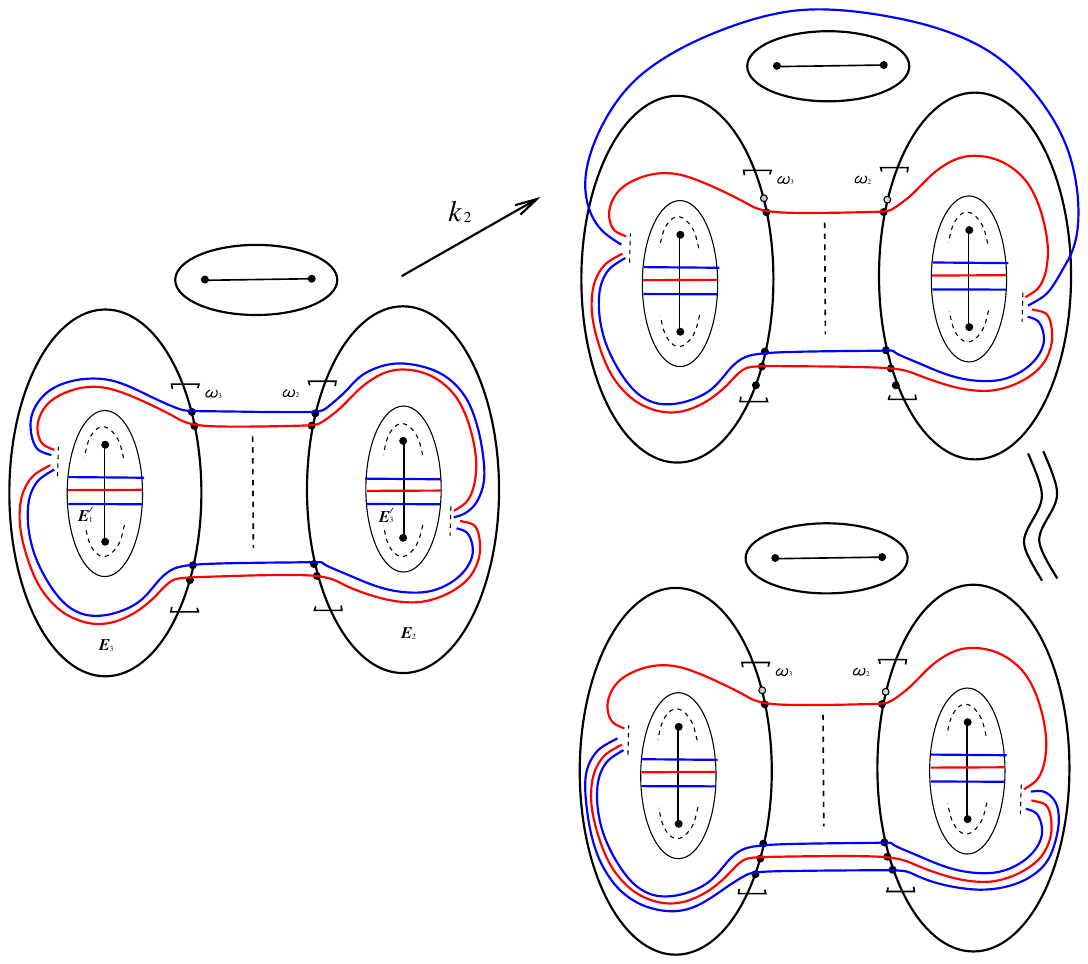}
\vskip -270pt
\caption{Normal jump move $k_2$ to the special case}\label{c36}
\end{figure}
\begin{proof}
Since $\boldsymbol{\beta}\cap \partial E_1=\emptyset$, $\beta_2\cup\beta_3$ meets only $\omega_2$ and $\omega_3$. Moreover, the four endpoints $3,4,5$ and $6$ of $\beta_2\cup\beta_3$ are in $E_1\cup E_3$. Therefore,  $|(\beta_2\cup\beta_3)\cap \omega_2|=|(\beta_2\cup\beta_3)\cap \omega_3|$. This implies that $p_2=p_3$.\\

Now, we claim that  $q_2+q_3$ is preserved for all minimal normal coordinates of $T$. Since  $\boldsymbol{\beta}$ is a normal coordinate of a rational $3$-tangle $T$ such that $\boldsymbol{\beta}\cap \partial E_1=\emptyset$, we can keep having a minimal normal coordinate of $T$ with respect to $\partial E_1$ by appying $k_1$ and $k_2$ alternatively. Moreover, the applying the normal jump moves preserves $p_2(=p_3)$. (Refer to Figure~\ref{c36}.)  By a similar argument of Lemma~\ref{T6}, each normal jump move changes $q_2$ and $q_3$ to $q_2\pm 1$ and $q_3\mp 1$. Therefore, $q_2+q_3$ would be preserved.

\end{proof}

\begin{Thm}
For a given rational $3$-tangle $T$, $\mathcal{N}(T)$ is contractible.

\end{Thm}
\begin{proof}
We assume that there exists a loop in $\mathcal{N}(T)$ for a contradiction. Let $v_1, v_2,..,v_{n}$ be the distinct sequence of normal forms of $T$ for the loop such that $v_{i+1}$ is obtained from $v_i$ by a normal jump move and $v_1$ is obtained from $v_{n}$ by a normal jump move, where $1\leq i\leq n-1$ and $n>2$. We assume that the sequence has a minimal number of terms(vertices) to make a loop.  Then, by Lemma~\ref{L5}, we note that the previous term and the next term of each term are obtained by $(k_s)^{-1}$ and $k_t$ for distinct $s,t\in\{1,2,3\}$. If not, then we can reduce the number of terms to have a different loop. Let $\boldsymbol{v}_1,\boldsymbol{v}_2,...,\boldsymbol{v}_{n}$ be the corresponding normal coordinates of $v_1, v_2,..,v_{n}$ respectively. We want to point out that the bridge arcs in the same position of the normal coordinates have the same endpoints. 
 We note that there are two cases for this as follows.
\vskip 10pt
\textbf{Case 1} : $|\boldsymbol{v}_i\cap \partial E_1|\neq |\boldsymbol{v}_{j}\cap \partial E_1|$ for some $i,j\in\{1,2,...,n\}$.\vskip 10pt
If  $|\boldsymbol{v}_i\cap \partial E_1|\neq |\boldsymbol{v}_{j}\cap \partial E_1|$ for some $i,j\in\{1,2,...,n\}$ then we have $|\boldsymbol{v}_1\cap \partial E_1|<|\boldsymbol{v}_{j}\cap \partial E_1|$ or $|\boldsymbol{v}_1\cap \partial E_1|>|\boldsymbol{v}_{j}\cap \partial E_1|$  for some $j\in\{2,...,n\}$. Theorem~\ref{T1} then implies that $|\boldsymbol{v}_1\cap \partial E_1|<|k_s(\boldsymbol{v}_{n-1})\cap \partial E_1|=|\boldsymbol{v}_{1}\cap \partial E_1|$ or $|\boldsymbol{v}_1\cap \partial E_1|>|k_s(\boldsymbol{v}_{n-1})\cap \partial E_1|=|\boldsymbol{v}_{1}\cap \partial E_1|$ for some $s\in\{1,2,3\}$.
This makes a contradiction for the construction of the loop.
\vskip 10pt
\textbf{Case 2} : $|\boldsymbol{v}_1\cap \partial E_1|=|\boldsymbol{v}_{i}\cap \partial E_1|$ for all $i\in\{2,...,n\}$.\vskip 10pt

 We note that three bridge arcs of $\boldsymbol{v}_1$ should meet with $\partial E_1$. Otherwise,  $\boldsymbol{v}_1,\boldsymbol{v}_2,...,\boldsymbol{v}_{n}$ cannot make a loop. There is no $\boldsymbol{v}_k$ so that $k_s(\boldsymbol{v}_{k-1})=\boldsymbol{v}_{k}$ for some $k\in\{2,...,n\}$ if $\beta_s$ is disjoint with $\partial E_1$ since $|\boldsymbol{v}_k\cap \partial E_1|\neq |\boldsymbol{v}_{k-1}\cap \partial E_1|$ by Lemma~\ref{T6}. 
We have two subcases $(1)$ and $(2)$ for Case $2$  without loss of generality. We note that it is impossible to have the case that $|\boldsymbol{v}_1\cap\partial E_1|<|k_s(\boldsymbol{v}_1)\cap\partial E_1|$ for all $s\in\{1,2,3\}$ to have a loop. If there is a loop then every $\boldsymbol{v}_i$ needs to be obtained from  $\boldsymbol{v}_{i-1}$ by $k_j^2$ for some $j$. However, because of Lemma~\ref{L0}, the number of elements for this loop is at most two. This makes a contradiction. \vskip 10pt

$(1)$~ $|\boldsymbol{v}_1\cap\partial E_1|=|k_1(\boldsymbol{v}_1)\cap\partial E_1|$ and $|\boldsymbol{v}_1\cap\partial E_1|<|k_i(\boldsymbol{v}_1)\cap\partial E_1|$ if $i\in\{2,3\}$.\vskip 10pt

By the assumption of this case, we note that $k_1(\boldsymbol{v}_{n})=\boldsymbol{v}_{1}$ to have a loop. If $k_j^2(\boldsymbol{v}_{n})=\boldsymbol{v}_{1}$ for some $j\in\{2,3\}$ then $\boldsymbol{v}_{2}=k_1(\boldsymbol{v}_{1})$ by Lemma~\ref{L0}. However, $|\boldsymbol{v}_1\cap\partial E_1|\neq |k_1(\boldsymbol{v}_1)\cap\partial E_1|$ because of the argument in Lemma~\ref{L-1}.  So, $(k_1)^{-1}(\boldsymbol{v}_{1})=k_1^s(\boldsymbol{v}_{1})=\boldsymbol{v}_{n}$ for some $s\in\{1,2\}$. 
Let $v_n=\{\beta_1',\beta_2,\beta_3\}$ be the normal form of $\boldsymbol{v}_{n}$. If there exist subarcs of $\beta_1'\cup\beta_2\cup\beta_3$ connecting $E_i$ and $E_j$ for all distinct $i,j\in\{1,2,3\}$ then $k_1^1(\boldsymbol{v}_{n})=\boldsymbol{v}_{1}$ and $k_j^2(\boldsymbol{v}_{1})=\boldsymbol{v}_{2}$ for some $j\in\{2,3\}$. This implies that $k_1^1(\boldsymbol{v}_{1})=\boldsymbol{v}_{n}$. However, by the argument in Lemma~\ref{L0} and Lemma~\ref{L-1}, we note that $k_1^1(\boldsymbol{v}_{1})\neq \boldsymbol{v}_{n}$ if $\boldsymbol{v}_{1}$ has the normal mapping $k_j^2$. This makes a contradiction. Therefore, it is impossible to have a loop. \vskip 10pt

$(2)$~ $|\boldsymbol{v}_1\cap\partial E_1|=|k_i(\boldsymbol{v}_1)\cap\partial E_1|$ if $i=1,2$ and $|\boldsymbol{v}_1\cap\partial E_1|<|k_3(\boldsymbol{v}_1)\cap\partial E_1|$.

\vskip 10pt
We first note that $\boldsymbol{v}_2$ and $\boldsymbol{v}_n$  cannot be obtained from $\boldsymbol{v}_1$ by $k_j^2$ for some $j\in\{1,2,3\}$ by the argument mentioned above. This implies that  $\boldsymbol{v}_2$ and $\boldsymbol{v}_n$  are obtained from $\boldsymbol{v}_1$ by $k_1^1$ and  $k_2^1$, or $k_2^1$ and  $k_1^1$ respectively. Now, we focus on $\boldsymbol{v}_2$. By the previous argument in this case, we also note that $|\boldsymbol{v}_2\cap\partial E_1|=|k_i(\boldsymbol{v}_2)\cap\partial E_1|$ if $i=1,2$ and $|\boldsymbol{v}_2\cap\partial E_1|<|k_3(\boldsymbol{v}_2)\cap\partial E_1|$. By continuing this procedure, we note that  $\boldsymbol{v}_{k+1}$ is obatined from $\boldsymbol{v}_k$ by $k_1^1$ or $k_2^1$ alternatingly for all $k\in\{1,2,...,n-1\}$. However, by Lemma~\ref{T6}, we cannot make a loop from  $\boldsymbol{v}_1,\boldsymbol{v}_2,...,\boldsymbol{v}_{n}$. 
The  alternating mappings of $k_1$ and  $k_2$ keep increasing (or decreasing) the connecting pattern $q_1$ by $1$ as in the diagram of Figure~\ref{c37}.  Therefore, it is impossible to have a loop if it satisfies $(2)$ and this completes the proof of the statement that $\mathcal{N}(T)$ is contractible.
\end{proof}

 \section{The representative of normal coordinates of $T$}\label{s6}

 In this section, we  discuss how to assign the representative of  normal coordinates of a rationtal $3$-tangle $T$ with a certain rule. It completes the classification of rational $3$-tangles as an effectiveness of the contractibility of the normal complex of a rational tangle $T$. Let $\boldsymbol{\beta}=(\beta_1,\beta_2,\beta_3)$  be a minimal normal coordinate of a rational $3$-tangle $T$ with respect to $\partial E_1$. We want to point out  that $\boldsymbol{\beta}$ is in minimal general position in terms of $\partial E$. We note that there is no $k_{i}^2$ such that $|k_i^2(\boldsymbol{\beta})\cap \partial E_1|=|\boldsymbol{\beta}\cap \partial E_1|$ for any $i\in\{1,2,3\}$. If there exists $k_i^2$ with $\boldsymbol{\beta}$ then we have either $|\boldsymbol{\beta}\cap \partial E|<|k_i^2(\boldsymbol{\beta})\cap \partial E|$ or  $|\boldsymbol{\beta}\cap \partial E|>|k_i^1(\boldsymbol{\beta})\cap \partial E|$ by the diagrams of Figure~\ref{outlet2}. Then, by Lemma~\ref{L-1}, we also have  either $|\boldsymbol{\beta}\cap \partial E_1|<|k_i^2(\boldsymbol{\beta})\cap \partial E_1|$ or  $|\boldsymbol{\beta}\cap \partial E_1|>|k_i^1(\boldsymbol{\beta})\cap \partial E_1|$. This makes a contradiction. Therefore, $k_i$ means $k_i^1$ with $\boldsymbol{\beta}$  which is a minimal normal coordinate of $T$. 

\begin{figure}[htb]
\includegraphics[scale=.75]{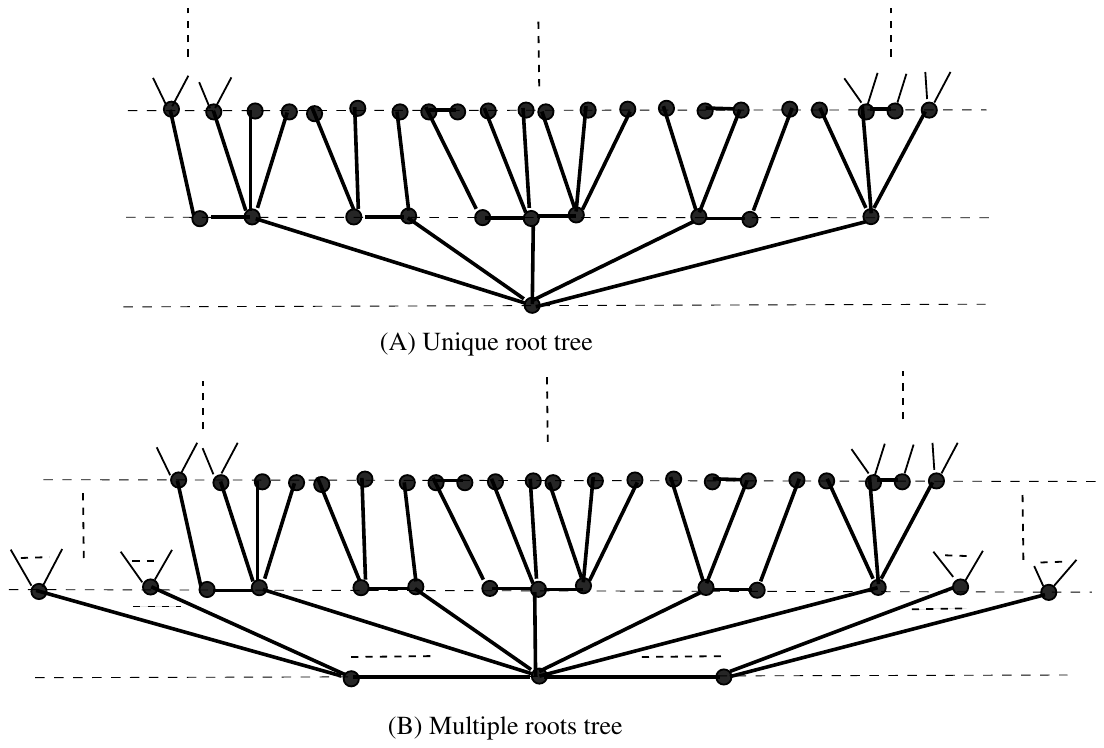}
\vskip -370pt
\caption{Hierarchy of trees for rational 3-tangles}\label{c9}
\end{figure}

First of all, we consider the case that $\boldsymbol{\beta}\cap \partial E_i= \emptyset$ for some $i$. Without loss of generality we assume that  $\boldsymbol{\beta}\cap \partial E_1= \emptyset.$  Recall that $E_1$ contains $\beta_1$ and $|(\beta_2\cup\beta_3)\cap \partial E_2|=|(\beta_2\cup\beta_3)\cap \partial E_3|$. Moreover, by Lemme~\ref{T6-1}, $p_2=p_3$ and $q_2+q_3=c$ for all minimal normal coordinates of $T$ with respect to $\partial E_1$, where $c$ is a fixed constant. In this case, we especially choose the minimal normal coordinate of $T$ so that $q_2=\displaystyle{\left[{(c+1)/ 2}\right]}$, where $[*]$ is the greatest integer function. Clear, $q_3=\displaystyle{\left[{c/ 2}\right]}$. We note that the choice is unique. So, this assigning is well-defined. \\

Secondly, we assume that $\boldsymbol{\beta}\cap \partial E_1\neq \emptyset.$ 
Then, we have the following three cases to assign the representative of $T$. Without loss of generality, we assume that a minimal normal coordinate $\boldsymbol{\beta}$ of $T$ with respect to $\partial E_1$ has the inequality that $|\boldsymbol{\beta}\cap \partial E_1|<|k_3(\boldsymbol{\beta})\cap \partial E_1|$  since it is impossible to have $|\boldsymbol{\beta}\cap \partial E_1|=|k_i(\boldsymbol{\beta})\cap \partial E_1|$ for all $i=1,2,3$ by Lemma~\ref{T3}.  Recall that the colors for $\beta_1,\beta_2$ and $\beta_3$ are blue, red and green respectively.\\

\textbf{Case 1 :}   $\boldsymbol{\beta} \cap \partial E_1\neq \emptyset$, $|\boldsymbol{\beta}\cap \partial E_1|<|k_i(\boldsymbol{\beta})\cap \partial E_1|$ for all $i=1,2,3$. \\

In this case, we assign the minimal normal coordinate $\boldsymbol{\beta}$ as the representative of normal coordinates of $T$. Since it is unique, it is well-defined.\\

\textbf{Case 2 :}   $\boldsymbol{\beta} \cap \partial E_1\neq \emptyset$, $|\boldsymbol{\beta}\cap \partial E_1|=|k_1(\boldsymbol{\beta})\cap \partial E_1|$ and $|\boldsymbol{\beta}\cap \partial E_1|<|k_i(\boldsymbol{\beta})\cap \partial E_1|$ for  $i=2,3$. \\

We note that $\boldsymbol{\beta}$ should have $\textbf{GRG}$ and $\textbf{GB-BG}$ type. (Refer to~Theorem~\ref{T5-1}.) We investigate $\textbf{GRG}$ type and $\textbf{GB-BG}$ type in terms of  the window $\omega_1$.  If the leftmost one is \textbf{GRG} type then we choose $\boldsymbol{\beta}$ as the representative of normal coordinates of $T$. Otherwise, we apply $k_1$ and $k_2$ alternatively to get all of minimal normal coordinates of $T$ with respect to $\partial E_1$ by Theorem~\ref{T5-1}. Among them, we choose the minimal normal coordinate of $T$ with respect to $\partial E_1$ which contains either the leftmost $\textbf{GRG}$ type or  the leftmost $\textbf{GB-BG}$ containing minimal intersections between the green dots. Recall that that the intersection number between a pair of consecutivce green dots for $\textbf{GB-BG}$ type during each process to find all of the minimal normal coordinates of $T$ with respect to $\partial E_1$ is decreased by $2$. Therefore, there exists unique minimal formal coordinate of $T$ with respect to $\partial E_1$ satisfying the condtion above. Therefore, the assigning is well-defined.\\

\textbf{Case 3:}    $\boldsymbol{\beta} \cap \partial E_1\neq \emptyset$, $|\boldsymbol{\beta}\cap \partial E_1|=|k_i(\boldsymbol{\beta})\cap \partial E_1|$ for $i=1,2$ and $|\boldsymbol{\beta}\cap \partial E_1|<|k_3(\boldsymbol{\beta})\cap \partial E_1|$.\\

By Theorem~\ref{T5}, all minimal normal coordinates of $T$ with respect to $\partial E_1$ have the same condition with $\boldsymbol{\beta}$ above. We have the two subcases as follows. \\

$\bullet$ \textbf{subcase 1:} $\beta_3\cap\partial E_1= \emptyset$. In this case, we assume that $\beta_j\not\in E_i$ for any $i, j\in\{1,2,3\}$ to exclude the case that $\boldsymbol{\beta}\cap \partial E_k=\emptyset$ for some $k$. We choose the minimal normal coordinate of $T$ with respect to $\partial E_1$ so that $q_1=0$. We note that it is unique by an argument in Lemma~\ref{T6} and this implies that the assigning is well-defined.\\

$\bullet$ \textbf{subcase 2:} $\beta_3\cap\partial E_1\neq \emptyset$. By Theorem~\ref{T5}, we note that $n_{\textbf{GRG}}=n_{\textbf{GBG}}=0$ and $n_{\textbf{GR-RG}}=n_{\textbf{GB-BG}}$. By an argument in Theorem~\ref{T1}, we note that every minimal normal coordinates of $T$ with respect to $\partial E_1$ have the same condition with $\boldsymbol{\beta}$. Moreover, they can be obtained from $\boldsymbol{\beta}$ by appying $k_1$ and $k_2$ alternatively for a certain number of times. Among them, we take the minimal normal coordinate of $T$ with respect to $\partial E_1$ so that the leftmost  ${\textbf{GR-RG}}$ or ${\textbf{GB-BG}}$ type has the minimal intersections between the two green dots. We note that it is unique by a similar argument in \textbf{Case 2} and this implies that the choice is well-defined.\\

 The following theorem   solves the last puzzle to classify rational $3$-tangles.
\begin{Thm}
Suppose that $\boldsymbol{\beta}=(\beta_1,\beta_2,\beta_3)$ and $\boldsymbol{\beta}'=(\beta_1',\beta_2',\beta_3')$  are the representatives of normal coordinates of $T$ and $T'$ with respect to $\partial E_1$ respectively with the regulation above. Let $(p_1, q_1, p_2, q_2, p_3, q_3)$ and $(p_1', q_1', p_2', q_2', p_3', q_3')$ be the Dehn's parameters of  $\boldsymbol{\beta}$ and $\boldsymbol{\beta}'$ respectively. Then $(p_1,q_1 , p_2, q_2, p_3, q_3)=(p_1', q_1', p_2', q_2', p_3', q_3')$ if and only if $T$ and $T'$ are isotopic.
\end{Thm}
\begin{proof}

It is clear that  if $(p_1,q_1 , p_2, q_2, p_3, q_3)=(p_1', q_1', p_2', q_2', p_3', q_3')$ then  $T$ and $T'$ are isotopic.\\

For the opppsite direction, we note that if $T$ and $T'$ are isotopic then their representatives of them with respect to $\partial E_1$ are the same by the arguments in \textbf{Case 1, Case 2} and \textbf{Case 3}. $\boldsymbol{\beta}'$  can be obtained from  $\boldsymbol{\beta}$ by a number of sequence of normal mappings, but  $\mathcal{N}(T)$ is contractible. This implies $(p_1,q_1 , p_2, q_2, p_3, q_3)=(p_1', q_1', p_2', q_2', p_3', q_3')$ and it completes the proof.

\end{proof}

 \section*{Acknowledgements}

 This research was supported by Basic Science Research Program through the National Research Foundation of Korea (NRF)  funded by the the Ministry of Education (RS-2023-00242412).


\begin{thebibliography}{1}
 
\bibitem{2} H. Cabrera,  {\em On the classification of rational 3-tangles},
J. Knot Theory Ramifications 12 (2003), no. 7, 921–946.

\bibitem{0} J. Conway, {\em An enumeration of knots and links, and some of their algebraic properties},
in Computational Problems in Abstract Algebra (Pergamon, Oxford, Press, 1970),
pp. 329–358.
\bibitem{1} B. Kwon, {\em An algorithm to classify rational $3$-tangles},
  J. Knot Theory Ramifications 24 (2015), no. 1, 1550004 (62 pages).
\bibitem{3} B. Kwon; J. H. Lee, {\em Normal form of rational $3$-tangles}, Topology and Its Applications, Volume 369 (2025), 109388



  


\end{thebibliography}
\end{document}